\documentclass[12pt]{article}

\usepackage{amssymb}
\usepackage{amsthm}
\usepackage{fleqn}
\usepackage{graphicx}
\usepackage{epstopdf}
\usepackage{amsfonts}
\usepackage[numbers]{natbib}
\usepackage[colorlinks,citecolor=blue,urlcolor=blue]{hyperref}
\usepackage{amsmath}
\usepackage{geometry}
\usepackage{array}

\newtheorem{theorem}{Theorem}[section]
\newtheorem{lemma}[theorem]{Lemma}

\newtheorem{remark}[theorem]{Remark}
\newtheorem{definition}[theorem]{Definition}

\begin{document}

\title{Robust optimal stopping with regime switching
\thanks{This work was supported by the National Key R\&D Program
of China (2023YFA1009200) and the National Natural Science Foundation of China
(61961160732, 12471418, 12171086).}
}

\author{Siyu Lv\thanks{School of Mathematics, Southeast University,
Nanjing 211189, China (lvsiyu@seu.edu.cn).}
\and
Zhen Wu\thanks{School of Mathematics, Shandong University,
Jinan 250100, China (wuzhen@sdu.edu.cn).}
\and
Jie Xiong\thanks{Department of Mathematics
and SUSTech International Center for Mathematics,
Southern University of Science and Technology, Shenzhen 518055, China
(xiongj@sustech.edu.cn).}
\and
Xin Zhang\thanks{School of Mathematics, Southeast University,
Nanjing 211189, China (x.zhang.seu@gmail.com).}
}


\maketitle

\begin{abstract}
In this paper, we study an optimal stopping problem in the presence of \emph{model uncertainty}
and \emph{regime switching}. The \emph{max-min formulation} for robust control and the
\emph{dynamic programming approach} are adopted to establish a general theoretical framework
for such kind of problem. First, based on the dynamic programming principle, the value function
of the optimal stopping problem is characterized as the unique \emph{viscosity solution} to the
associated Hamilton-Jacobi-Bellman equation. Then, the so-called \emph{smooth-fit principle}
for optimal stopping problems is proved in the current context, and a verification theorem
consisting of a set of \emph{sufficient conditions} for robust optimality is established.
Moreover, when the Markov chain has a large state space and exhibits a \emph{two-time-scale}
structure, a singular perturbation approach is utilized to reduce the complexity involved and
obtain an \emph{asymptotically optimal} solution. Finally, an example of choosing the best time
to sell a stock is provided, in which numerical experiments are reported to illustrate the
implications of model uncertainty and regime switching.
\end{abstract}


\section{Introduction}

Optimal stopping refers to the problem of finding the best choice among all possible
stopping times, based on averaged or expected results. To study optimal stopping problems,
the \emph{verification theorem approach} has been extensively employed because it provides
some sufficient conditions that are easy to verify and typically leads to certain ordinary
or partial differential equations (in the form of \emph{variational inequalities}) that
can be solved analytically or numerically; see Lions \cite{Lions1975,Lions1976},
Bensoussan and Lions \cite[Chapter 3]{BL1982}, Menaldi \cite{Menaldi1980stopping,Menaldi1982},
{\O}ksendal \cite[Chapter 10]{Oksendal2003}, and Pham \cite[Chapter 5]{Pham2009}.
Apart from the interest in its own right, optimal stopping enjoys many applications in various fields,
such as option pricing (Peskir and Shiryaev \cite[Chapter \uppercase\expandafter{\romannumeral7}]{PS2006}),
stock selling ({\O}ksendal \cite[Examples 10.2.2 and 10.4.2]{Oksendal2003}),
and quickest detection (see the survey \cite{Shiryaev2010} by Shiryaev
and the recent works \cite{EM2023MOR,EMP2024SICON,EP2022AAP} by Ernst, Mei, and Peskir).

The robust stochastic control is a framework that deals with decision-making when, besides
risk in standard stochastic control, there exists ambiguity (or, model uncertainty) as well.
Briefly speaking, risk means the randomness involved in \emph{a given model}, while ambiguity
implies the uncertainty among \emph{a set of models}. A popularly adopted formulation for robust
stochastic control is the so-called max-min paradigm, i.e., finding an optimal control under the
\emph{worst-case} scenario; see \cite{HS2001,Pun2018,CFT2013,TF2013}. On the other hand, motivated
by the need of more realistic models that better reflect random environment, the regime switching
diffusion has attracted increasing interest in recent years. Mathematically, the regime switching
diffusion is a \emph{two-component} process $(X_{t},\alpha_{t})$ in which $X_{t}$ evolves according to
a diffusion process whose coefficients depend on the regime of a finite-state Markov chain $\alpha_{t}$.
As a result, the regime switching diffusion can capture more directly the \emph{discrete events}
that are less frequent but more significant to longer-term system behavior, and has been found to be
very useful in modeling and simulation (see \cite{GZ2005,YaoZhangZhou2006,YinZhang2013}).

This paper considers an optimal stopping problem involving model uncertainty and regime switching.
Following the max-min formulation and based on the dynamic programming principle (DPP), the value
function is characterized to be the unique viscosity solution to the associated Hamilton-Jacobi-Bellman
(HJB) equation; see Section \ref{Section VS}. With the help of the viscosity solution theory,
we proceed to prove the smooth-fit property of the value function through the boundary of the
so-called \emph{continuation region}; see Section \ref{Section SFP}. Then, a verification theorem
as a sufficient criterion that can be used to construct a robust optimal stopping rule is established
in Section \ref{Section VT}. Quite often, the states of the Markov chain belong to several
groups so that the chain jumps rapidly \emph{within} each group and slowly \emph{between}
the groups (see Yin and Zhang \cite{YinZhang2013}). In this case, we employ a singular perturbation
approach based on a two-time-scale model to reduce the complexity involved; see Section \ref{Section TTS}.
Finally, we provide an example of a stock selling problem, in which numerical experiments are conducted
to demonstrate the theoretical results; see Section \ref{Section SS}.

The main novelty and contribution of this paper include three aspects:
(\romannumeral1) We present a comprehensive study for the optimal stopping problem with model
uncertainty and regime switching, including the viscosity solution characterization of HJB
equation, the smooth-fit principle, and the verification theorem. Compared with traditional
optimal stopping problems (e.g., \cite[Chapter 10]{Oksendal2003} and \cite[Chapter 5]{Pham2009}),
the proofs are more delicate with some significant technical difficulties to be overcome.
(\romannumeral2) When the Markov chain has a two-time-scale structure, we prove the convergence of
the original problem to a limit problem with \emph{much simpler} structure, and the solution of
the limit problem provides a \emph{close approximation} to that of the original problem.
(\romannumeral3) The notion of viscosity solution is fully utilized in this paper and shows its
good \emph{versatility} for optimal control problems. In the Definition \ref{VS Definition 1}
of viscosity solution, the ``\emph{global} maximum (resp., minimum)" is adopted. It may be replaced
by ``\emph{local} maximum (resp., minimum)" or even ``\emph{strict local} maximum (resp., minimum)".
These three terms are essentially \emph{equivalent} (see Pham \cite[Chapter 4]{Pham2009}) and will
be used in characterization of HJB equation, smooth-fit principle, and two-time-scale convergence,
respectively, to facilitate the corresponding proofs.

\section{Problem formulation}\label{Section PF}

Let $\mathbb{R}$ be the real space. Denote $C^{1}$ (resp., $C^{2}$) the space of (resp., twice)
continuously differentiable functions. Let $(\Omega,\mathcal{F},\mathbb{P})$ be a probability
space on which a 1-dimensional standard Brownian motion $B^{\mathbb{P}}_{t}$, $t\geq 0$, and
a Markov chain $\alpha_{t}$, $t\geq 0$, are defined. Assume that $B_{t}^{\mathbb{P}}$ and
$\alpha_{t}$ are independent. The Markov chain $\alpha_{t}$ takes values in a finite state space
$\mathcal{M}=\{1,\ldots,m\}$. Denote $Q=(\lambda_{ij})_{i,j\in\mathcal{M}}$ the generator
(i.e., the matrix of transition rates) of $\alpha_{t}$.
Let $\{\mathcal{F}_{t}\}_{t\geq 0}$ be the natural filtration of $B_{t}^{\mathbb{P}}$ and $\alpha_{t}$.

The \emph{reference model} (or, the original state equation) is given by the following
regime switching diffusion:
\begin{equation}\label{system P}
\left\{
\begin{aligned}
dX_{t}=&b(X_{t},\alpha_{t})dt+\sigma(X_{t},\alpha_{t})dB^{\mathbb{P}}_{t},\quad t\geq0,\\
X_{0}=&x\in \mathbb{R},\quad \alpha_{0}=i\in \mathcal{M},
\end{aligned}
\right.
\end{equation}
where $b,\sigma:\mathbb{R}\times\mathcal{M}\mapsto \mathbb{R}$ are two functions satisfying:

(A1) There exists a positive constant $L_{1}$ such that, for any $x_{1},x_{2}\in \mathbb{R}$
and $i\in \mathcal{M}$,
\begin{equation*}
\begin{aligned}
|b(x_{1},i)-b(x_{2},i)|+|\sigma(x_{1},i)-\sigma(x_{2},i)|\leq L_{1}|x_{1}-x_{2}|.
\end{aligned}
\end{equation*}
Under Assumption (A1), the reference model (\ref{system P}) admits a unique strong solution
(see Yin and Zhang \cite{YinZhang2013}).

The \emph{alternative models} are induced by a set of probability measures
$\mathcal{P}=\{\mathbb{Q}: \mathbb{Q}\sim \mathbb{P}\}$ that are \emph{equivalent}
to $\mathbb{P}$. By Girsanov's transformation for regime switching diffusions
(see Yao et al. \cite[Lemma 1]{YaoZhangZhou2006}), for each $\mathbb{Q}\in \mathcal{P}$,
there exists an $\mathbb{R}$-valued $\mathcal{F}_{t}$-adapted process (i.e.,
the Radon-Nikodym derivative) $q_{t}$, $t\geq0$, such that
\begin{equation*}
\begin{aligned}
\frac{d\mathbb{Q}}{d\mathbb{P}}\bigg|_{\mathcal{F}_{t}}
=\text{exp}\bigg(\int_{0}^{t}q_{s}dB_{s}^{\mathbb{P}}
-\frac{1}{2}\int_{0}^{t}q_{s}^{2}ds\bigg)
\doteq\mathcal{E}_{t}^{q}.
\end{aligned}
\end{equation*}
Further, if $\mathcal{E}_{t}^{q}$ satisfies
\begin{equation}\label{admissible q condition}
\begin{aligned}
\mathbb{E}^{\mathbb{P}}[\mathcal{E}_{t}^{q}]=1,\quad t\geq0,
\end{aligned}
\end{equation}
then $\mathcal{E}_{t}^{q}$ is a $\mathbb{P}$-martingale, and under $\mathbb{Q}$, the process
defined by $B_{t}^{\mathbb{Q}}=B_{t}^{\mathbb{P}}-\int_{0}^{t}q_{s}ds$, $t\geq0$, is a standard
Brownian motion, $\alpha_{t}$ remains to be a Markov chain with the \emph{original} generator $Q$,
and $B_{t}^{\mathbb{Q}}$ and $\alpha_{t}$ are still independent. Note that the model uncertainty
between an alternative model and the reference model (i.e., the difference between $\mathbb{Q}$
and $\mathbb{P}$) is characterized by $q_{t}$, which can be regarded as a \emph{control variable}.
We then denote the set of all admissible controls as
$\mathcal{U}=\{q_{t}:q_{t}\text{ satisfies }(\ref{admissible q condition})\}$.
We also denote the set of all $\mathcal{F}_{t}$-stopping times as $\mathcal{S}$.

Given a new probability measure $\mathbb{Q}$ associated with a $q_{t}\in\mathcal{U}$,
the reference model (\ref{system P}) can be rewritten as
\begin{equation}\label{system Q}
\left\{
\begin{aligned}
dX_{t}=&[b(X_{t},\alpha_{t})+\sigma(X_{t},\alpha_{t})q_{t}]dt
+\sigma(X_{t},\alpha_{t})dB^{\mathbb{Q}}_{t},\\
X_{0}=&x\in \mathbb{R},\quad \alpha_{0}=i\in \mathcal{M}.
\end{aligned}
\right.
\end{equation}
In what follows, we will consider the alternative model (\ref{system Q}) instead of (\ref{system P}),
and the solution to (\ref{system Q}) is sometimes denoted by $X_{t}^{q}$ if emphasis on $q_{t}$ is needed.

For $(\tau,q_{t})\in\mathcal{S}\times\mathcal{U}$, the reward functional
for the optimal stopping problem is defined as
\begin{equation*}
\begin{aligned}
J(x,i;\tau,q_{t})=\mathbb{E}^{\mathbb{Q}}\bigg[\int_{0}^{\tau}e^{-rt}
\bigg(f(X_{t},\alpha_{t})+\Theta(q_{t})\bigg)dt
+e^{-r\tau}g(X_{\tau},\alpha_{\tau})\bigg],
\end{aligned}
\end{equation*}
where $r>0$ is the discount factor, $f,g:\mathbb{R}\times\mathcal{M}\mapsto \mathbb{R}$
are two functions satisfying:

(A2) There exists a positive constant $L_{2}$ such that,
for any $x_{1},x_{2}\in \mathbb{R}$ and $i\in \mathcal{M}$,
\begin{equation*}
\begin{aligned}
|f(x_{1},i)-f(x_{2},i)|+|g(x_{1},i)-g(x_{2},i)|\leq L_{2}|x_{1}-x_{2}|,
\end{aligned}
\end{equation*}
and $\Theta:\mathbb{R}\mapsto \mathbb{R}$ is a continuous function to measure the difference
between $\mathbb{Q}$ and $\mathbb{P}$, or, the \emph{ambiguity aversion} of an investor
in financial terminology. Then, the value function of the optimal stopping problem is defined as
\begin{equation}\label{value function}
\begin{aligned}
V(x,i)=&\sup_{\tau\in\mathcal{S}}\widetilde{J}(x,i;\tau)
\doteq\sup_{\tau\in\mathcal{S}}\inf_{q_{t}\in\mathcal{U}}J(x,i;\tau,q_{t}).
\end{aligned}
\end{equation}
For further analysis, we also assume:

(A3) The value function $V(x,i)$ is Lipschitz continuous with respect to $x$.
\begin{remark}
Some sufficient conditions to make (A3) hold include that, for example,
the admissible controls are restricted to those $q_{t}$ that are uniformly
bounded to confine the scope of ambiguity degree, or the diffusion coefficient
$\sigma$ is independent of $x$, i.e., $\sigma(x,i)\equiv\sigma(i)$.
Moreover, Assumption (A3) also implies that $V(x,i)$ has at most linear growth
since $\mathcal{M}$ is a finite set.
\end{remark}

\section{Viscosity solution}\label{Section VS}

We first present the DPP for our optimal stopping problem, whose proof is omitted here;
one is referred to Pham \cite[Chapter 5]{Pham2009} for details.
\begin{theorem}\label{Theorem DPP}
Let Assumptions (A1)-(A3) hold. Then, for any stopping time $\nu$, we have
\begin{equation}\label{DPP}
\begin{aligned}
V(x,i)
=\sup_{\tau\in\mathcal{S}}\inf_{q_{t}\in\mathcal{U}}\mathbb{E}^{\mathbb{Q}}
\bigg[&\int_{0}^{\tau\wedge\nu}e^{-rt}\bigg(f(X_{t},\alpha_{t})+\Theta(q_{t})\bigg)dt\\
&+e^{-r\tau}g(X_{\tau},\alpha_{\tau})1_{\{\tau<\nu\}}
+e^{-r\nu}V(X_{\nu},\alpha_{\nu})1_{\{\tau\geq\nu\}}\bigg].
\end{aligned}
\end{equation}
\end{theorem}
Based on the DPP, we have the following HJB equation in the form of variational inequality
(see \cite{Lions1975,Lions1976,Menaldi1980stopping,Menaldi1982}):
\begin{equation}\label{HJB}
\begin{aligned}
\min\bigg\{\sup_{q\in \mathbb{R}}\bigg\{rv(x,i)-\mathcal{L}^{q}v(x,i)-Qv(x,\cdot)(i)
-f(x,i)-\Theta(q)\bigg\},
v(x,i)-g(x,i)\bigg\}=0,
\end{aligned}
\end{equation}
where $v(x,i):\mathbb{R}\times\mathcal{M}\mapsto \mathbb{R}$ is the solution,
$Qv(x,\cdot)(i)\doteq\sum_{j\neq i}\lambda_{ij}[v(x,j)-v(x,i)]\equiv\sum_{j=1}^{m}\lambda_{ij}v(x,j)$
is the \emph{infinitesimal operator} of the Markov chain, and
\begin{equation*}
\begin{aligned}
\mathcal{L}^{q}v(x,i)\doteq[b(x,i)+\sigma(x,i)q]v^{\prime}(x,i)
+\frac{1}{2}\sigma^{2}(x,i)v^{\prime\prime}(x,i).
\end{aligned}
\end{equation*}
\begin{definition}\label{VS Definition 1}
A continuous function $v(x,i)$, $i\in\mathcal{M}$, is called a viscosity subsolution
(resp., supersolution) to (\ref{HJB}) if for any $i\in\mathcal{M}$, we have
\begin{equation*}
\begin{aligned}
\min\bigg\{\sup_{q\in \mathbb{R}}\bigg\{rv(\overline{x},i)
-\mathcal{L}^{q}\varphi(\overline{x},i)-Qv(\overline{x},\cdot)(i)
-f(\overline{x},i)-\Theta(q)\bigg\},
v(\overline{x},i)-g(\overline{x},i)\bigg\}\leq0
\end{aligned}
\end{equation*}
(resp., $\geq0$), whenever $\varphi(x)\in C^{2}$ and $v(x,i)-\varphi(x)$ has a global maximum
(resp., minimum) at $x=\overline{x}\in \mathbb{R}$. Further, $v(x,i)$ is called a viscosity
solution if it is both a viscosity subsolution and supersolution.
\end{definition}
Moreover, an \emph{equivalent} definition of viscosity solution via the so-called
\emph{superdifferential} $J^{2,+}\phi(x)$ and \emph{subdifferential} $J^{2,-}\phi(x)$
of a function $\phi$ at $x$ (see Crandall et al. \cite[Section 2]{CIL1992} for
precise definitions) is given below.
\begin{definition}\label{VS Definition 2}
A continuous function $v(x,i)$, $i\in\mathcal{M}$, is called a viscosity subsolution
(resp., supersolution) to (\ref{HJB}) if for any $(x,i)\in \mathbb{R}\times\mathcal{M}$
and any $(\eta,\Lambda)\in J^{2,+}v(x,i)$ (resp., $J^{2,-}v(x,i)$), we have
\begin{equation*}
\begin{aligned}
\min\bigg\{\sup_{q\in \mathbb{R}}\bigg\{rv(x,i)-[b(x,i)+\sigma(x,i)q]\eta
-\frac{1}{2}\sigma^{2}(x,i)\Lambda-Qv(x,\cdot)(i)-f(x,i)-\Theta(q)\bigg\},&\\
v(x,i)-g(x,i)\bigg\}\leq0&
\end{aligned}
\end{equation*}
(resp., $\geq0$). Further, $v(x,i)$ is called a viscosity solution
if it is both a viscosity subsolution and supersolution.
\end{definition}
\begin{remark}
The upper (resp., lower) semicontinuity is enough to define the viscosity subsolution
(resp., supersolution) in Definitions \ref{VS Definition 1} and \ref{VS Definition 2};
see Pham \cite[Chapter 4]{Pham2009}. Here, we use continuity just for simplicity
of presentation since the value function is continuous.
\end{remark}
\begin{theorem}\label{existence theorem}
Let Assumptions (A1)-(A3) hold. Then, the value function $V(x,i)$, $i\in\mathcal{M}$,
defined by (\ref{value function}) is a viscosity solution to the HJB equation (\ref{HJB}).
\end{theorem}
\noindent{\it Proof.}
\emph{Viscosity supersolution property.}
For any fixed $i\in\mathcal{M}$, let $\varphi(x)\in C^{2}$ be such that $V(x,i)-\varphi(x)$
attains its global minimum at $x=\overline{x}$. On the one hand, taking $\tau=0$ in the
DPP (\ref{DPP}), we have
\begin{equation}\label{V>=g}
\begin{aligned}
V(x,i)\geq g(x,i),\quad i\in\mathcal{M}.
\end{aligned}
\end{equation}
On the other hand, based on $V$ and $\varphi$, we introduce a function
$\psi:\mathbb{R}\times\mathcal{M}\rightarrow \mathbb{R}$ defined as
\begin{equation}\label{psi}
\begin{aligned}
\psi(x,j)=\left\{
\begin{aligned}
&\varphi(x)+V(\overline{x},i)-\varphi(\overline{x}),\quad &&j=i,\\
&V(x,j),\quad &&j\neq i.
\end{aligned}
\right.
\end{aligned}
\end{equation}
Let $\nu=\tau_{\alpha}\wedge c$ for some $c>0$, where $\tau_{\alpha}$ is the first jump time
of $\alpha_{t}$. Taking an arbitrary $q_{t}\in\mathcal{U}$ and applying It\^{o}'s formula to
$e^{-rt}\psi(X_{t}^{q},\alpha_{t})$ between 0 and $\nu$, we obtain
\begin{equation*}
\begin{aligned}
&\mathbb{E}^{\mathbb{Q}}[e^{-r\nu}\psi(X_{\nu}^{q},\alpha_{\nu})]-\psi(\overline{x},i)\\
=&\mathbb{E}^{\mathbb{Q}}\bigg[\int_{0}^{\nu}e^{-rt}\bigg(-r\psi(X_{t}^{q},i)
+\mathcal{L}^{q}\varphi(X_{t}^{q})+Q\psi(X_{t}^{q},\cdot)(i)\bigg)dt\bigg].
\end{aligned}
\end{equation*}
Note that $V(x,i)-\varphi(x)$ attains its global minimum at $x=\overline{x}$,
so for $0\leq t<\nu$, it follows that $\psi(X_{t}^{q},i)=\varphi(X_{t}^{q})
+(V(\overline{x},i)-\varphi(\overline{x}))\leq V(X_{t}^{q},i)$.
Then we get
\begin{equation}\label{V>=}
\begin{aligned}
&\mathbb{E}^{\mathbb{Q}}[e^{-r\nu}V(X_{\nu}^{q},\alpha_{\nu})]-V(\overline{x},i)\\
\geq&\mathbb{E}^{\mathbb{Q}}\bigg[\int_{0}^{\nu}e^{-rt}\bigg(-rV(X_{t}^{q},i)
+\mathcal{L}^{q}\varphi(X_{t}^{q})+QV(X_{t}^{q},\cdot)(i)\bigg)dt\bigg].
\end{aligned}
\end{equation}
Choosing $\tau=\nu$ in the DPP (\ref{DPP}), we know that, for any $\varepsilon>0$,
there exists a $\overline{q}_{t}\in\mathcal{U}$ such that
\begin{equation}\label{DPP inf}
\begin{aligned}
V(\overline{x},i)+c\varepsilon
>\mathbb{E}^{\overline{\mathbb{Q}}}\bigg[\int_{0}^{\nu}e^{-rt}
\bigg(f(X_{t}^{\overline{q}},i)+\Theta(\overline{q}_{t})\bigg)dt
+e^{-r\nu}V(X_{\nu}^{\overline{q}},\alpha_{\nu})\bigg].
\end{aligned}
\end{equation}
Replacing $q_{t}$ by $\overline{q}_{t}$ in (\ref{V>=}) (noting the fact that $\nu$
is independent of $q_{t}$) and combining (\ref{V>=}) and (\ref{DPP inf}), we have
\begin{equation*}
\begin{aligned}
\mathbb{E}^{\overline{\mathbb{Q}}}\bigg[\int_{0}^{\nu}e^{-rt}
\bigg(rV(X_{t}^{\overline{q}},i)-\mathcal{L}^{\overline{q}}\varphi(X_{t}^{\overline{q}},i)
-QV(X_{t}^{\overline{q}},\cdot)(i)-f(X_{t}^{\overline{q}},i)-\Theta(\overline{q}_{t})\bigg)dt\bigg]
>-c\varepsilon.
\end{aligned}
\end{equation*}
Dividing both sides of the above inequality by $c$, letting $c\rightarrow0$,
and taking the supremum over $q\in \mathbb{R}$, we have
$\sup_{q\in \mathbb{R}}\{rV(\overline{x},i)-\mathcal{L}^{q}\varphi(\overline{x},i)
-QV(\overline{x},\cdot)(i)-f(\overline{x},i)-\Theta(q)\}\geq0$,
which together with (\ref{V>=g}) yield the desired viscosity supersolution property.

\emph{Viscosity subsolution property.}
If it were not, there would exist an $\overline{x}$, an $i$, and a $\varphi(x)\in C^{2}$
such that $V(x,i)-\varphi(x)$ attains its global maximum at $\overline{x}$, but
$\sup_{q\in \mathbb{R}}\{rV(\overline{x},i)-\mathcal{L}^{q}\varphi(\overline{x},i)
-QV(\overline{x},\cdot)(i)-f(\overline{x},i)-\Theta(q)\}>0$
and $V(\overline{x},i)-g(\overline{x},i)>0$. Then, there should exist a constant
$\overline{q}\in \mathbb{R}$ and a $\delta>0$ such that
\begin{equation}\label{subsolution proof 1}
\begin{aligned}
rV(X_{t}^{\overline{q}},i)-\mathcal{L}^{\overline{q}}\varphi(X_{t}^{\overline{q}},i)
-QV(X_{t}^{\overline{q}},\cdot)(i)-f(X_{t}^{\overline{q}},i)-\Theta(\overline{q})\geq\delta
\end{aligned}
\end{equation}
and
\begin{equation}\label{subsolution proof 2}
\begin{aligned}
V(X_{t}^{\overline{q}},i)-g(X_{t}^{\overline{q}},i)\geq\delta
\end{aligned}
\end{equation}
before $\nu\doteq\tau_{\delta}\wedge\tau_{\alpha}$, where $\tau_{\delta}$
is the first exit time of $X^{\overline{q}}$ from the neighborhood
$B_{\delta}(\overline{x})\doteq\{x\in \mathbb{R}:|x-\overline{x}|<\delta\}$.

For any $\tau\in \mathcal{S}$ and the function $\psi$ defined by (\ref{psi}),
applying It\^{o}'s formula to $e^{-rt}\psi(X_{t}^{\overline{q}},\alpha_{t})$
between $0$ and $\tau\wedge\nu$, substituting $V$ for $\psi$ as in the proof
of viscosity supersolution property, and noting (\ref{subsolution proof 1})
and (\ref{subsolution proof 2}), we have
\begin{equation}\label{subsolution proof 5}
\begin{aligned}
V(\overline{x},i)\geq&\mathbb{E}^{\overline{\mathbb{Q}}}
\bigg[\int_{0}^{\tau\wedge\nu}e^{-rt}\bigg(f(X_{t}^{\overline{q}},i)+\Theta(\overline{q})\bigg)dt\\
&+e^{-r\tau}g(X_{\tau}^{\overline{q}},\alpha_{\tau})1_{\{\tau<\nu\}}
+e^{-r\nu}V(X_{\nu}^{\overline{q}},\alpha_{\nu})1_{\{\tau\geq\nu\}}\bigg]\\
&+\delta \mathbb{E}^{\overline{\mathbb{Q}}}\bigg[\int_{0}^{\tau\wedge\nu}e^{-rt}dt+e^{-r\tau}1_{\{\tau<\nu\}}\bigg],
\end{aligned}
\end{equation}
where we can show that the last term in (\ref{subsolution proof 5}) is not less than
$\delta c_{0}$ for some $c_{0}>0$. In fact, similar to Pham \cite[Chapter 5]{Pham2009},
we choose a smooth function
$$
\phi(x)=c_{1}(1-|x-\overline{x}|^{2}/\delta^{2})
$$
with
\begin{equation*}
\begin{aligned}
0<c_{1}=\min\bigg\{1,
\bigg(r+\frac{2}{\delta}\max_{i\in\mathcal{M}}\sup_{x\in B_{\delta}(\overline{x})}
|b(x,i)+\sigma(x,i)\overline{q}|
+\frac{1}{\delta^{2}}\max_{i\in\mathcal{M}}\sup_{x\in B_{\delta}(\overline{x})}
\sigma^{2}(x,i)\bigg)^{-1}\bigg\}.
\end{aligned}
\end{equation*}
It can be seen that $\phi(\cdot)$ has the following properties:
(a) $\phi(\overline{x})=c_{1}$,
(b) $\phi(x)\leq c_{1}\leq 1$, $\forall x\in B_{\delta}(\overline{x})$,
(c) $r\phi(x)-\mathcal{L}^{\overline{q}}\phi(x)\leq1$, $\forall x\in B_{\delta}(\overline{x})$,
(d) $\phi(x)=0$, $\forall x\in \partial B_{\delta}(\overline{x})$.

By applying It\^{o}'s formula to $e^{-rt}\phi(X^{\overline{q}}_{t})$
between $0$ and $\tau\wedge\nu$, we have
\begin{equation*}
\begin{aligned}
c_{1}=\phi(\overline{x})
=&\mathbb{E}^{\overline{\mathbb{Q}}}\bigg[\int_{0}^{\tau\wedge\nu}e^{-rt}
\bigg(r\phi(X^{\overline{q}}_{t})-\mathcal{L}^{\overline{q}}\phi(X^{\overline{q}}_{t})\bigg)dt
+e^{-r\tau\wedge\nu}\phi(X^{\overline{q}}_{\tau\wedge\nu})\bigg]\\
\leq&\mathbb{E}^{\overline{\mathbb{Q}}}\bigg[\int_{0}^{\tau\wedge\nu}e^{-rt}dt
+e^{-r\tau}\phi(X^{\overline{q}}_{\tau})1_{\{\tau<\nu\}}
+e^{-r\nu}\phi(X^{\overline{q}}_{\nu})1_{\{\tau\geq\nu\}}\bigg]\\
\leq&\mathbb{E}^{\overline{\mathbb{Q}}}\bigg[\int_{0}^{\tau\wedge\nu}e^{-rt}dt
+e^{-r\tau}1_{\{\tau<\nu\}}+e^{-r\nu}\phi(X^{\overline{q}}_{\nu})1_{\{\tau\geq\nu\}}\bigg].
\end{aligned}
\end{equation*}
Note that
\begin{equation*}
\begin{aligned}
\mathbb{E}^{\overline{\mathbb{Q}}}[e^{-r\nu}\phi(X^{\overline{q}}_{\nu})1_{\{\tau\geq\nu\}}]
=\mathbb{E}^{\overline{\mathbb{Q}}}[e^{-r\tau_{\delta}}\phi(X^{\overline{q}}_{\tau_{\delta}})
1_{\{\tau\geq\nu\}}1_{\{\tau_{\delta}<\tau_{\alpha}\}}]
+\mathbb{E}^{\overline{\mathbb{Q}}}[e^{-r\tau_{\alpha}}\phi(X^{\overline{q}}_{\tau_{\alpha}})
1_{\{\tau\geq\nu\}}1_{\{\tau_{\delta}\geq\tau_{\alpha}\}}],
\end{aligned}
\end{equation*}
where the first term of the right-hand side is equal to 0
since $X^{\overline{q}}_{\tau_{\delta}}\in \partial B_{\delta}(\overline{x})$
and the second term
$
\mathbb{E}^{\overline{\mathbb{Q}}}[e^{-r\tau_{\alpha}}\phi(X^{\overline{q}}_{\tau_{\alpha}})
1_{\{\tau\geq\nu\}}1_{\{\tau_{\delta}\geq\tau_{\alpha}\}}]\doteq c_{2}<c_{1},
$
due to $\tau_{\alpha}>0$, a.s. Therefore,
\begin{equation*}
\begin{aligned}
\mathbb{E}^{\overline{\mathbb{Q}}}\bigg[\int_{0}^{\tau\wedge\nu}e^{-rt}dt
+e^{-r\tau}1_{\{\tau<\nu\}}\bigg]
\geq c_{1}-c_{2}\doteq c_{0}>0.
\end{aligned}
\end{equation*}
Going back to (\ref{subsolution proof 5}), it follows from the arbitrariness of $\tau$ that
\begin{equation*}
\begin{aligned}
V(\overline{x},i)
\geq&\sup_{\tau\in\mathcal{S}}\inf_{q_{t}\in\mathcal{U}}
\mathbb{E}^{\mathbb{Q}}\bigg[\int_{0}^{\tau\wedge\nu}e^{-rt}\bigg(f(X_{t}^{q},i)+\Theta(q_{t})\bigg)dt\\
&+e^{-r\tau}g(X_{\tau}^{q},\alpha_{\tau})1_{\{\tau<\nu\}}
+e^{-r\nu}V(X_{\nu}^{q},\alpha_{\nu})1_{\{\tau\geq\nu\}}\bigg]+\delta c_{0},
\end{aligned}
\end{equation*}
which leads to a contradiction to the DPP (\ref{DPP}).
~\hfill $\Box$

In the following uniqueness theorem, the function $\Theta(q)$ is specified to be
the \emph{relative entropy} form $\Theta(q)=q^{2}/2\theta$, which is commonly used
in the robust control literature (see \cite{HS2001,Pun2018}), and the positive
parameter $\theta$ is the so-called \emph{ambiguity factor}. In this case, we can solve
$q^{*}(x,i)\doteq\arg\sup_{q\in \mathbb{R}}H^{q}v(x,i)=-\theta\sigma(x,i)v^{\prime}(x,i)$,
and the HJB equation (\ref{HJB}) reduces to
\begin{equation}\label{uniqueness HJB}
\begin{aligned}
\min\bigg\{rv-bv^{\prime}+\frac{\theta}{2}\sigma^{2}[v^{\prime}]^{2}
-\frac{1}{2}\sigma^{2}v^{\prime\prime}-Qv(x,\cdot)(i)-f,
v-g\bigg\}=0,
\end{aligned}
\end{equation}
where the arguments $(x,i)$ are dropped for simplicity.
\begin{theorem}\label{uniqueness theorem}
Let Assumptions (A1)-(A3) hold. Let $V_{1}(x,i)$ and $V_{2}(x,i)$ be two viscosity solutions
to the HJB equation (\ref{uniqueness HJB}) and both have at most linear growth. Then, we have
$V_{1}(x,i)=V_{2}(x,i)$ for any $(x,i)\in \mathbb{R}\times\mathcal{M}$.
\end{theorem}
\noindent{\it Proof.}
We consider the following function defined on $\mathbb{R}\times \mathbb{R}\times\mathcal{M}$:
$$
\Psi(x_{1},x_{2},i)=V_{1}(x_{1},i)-V_{2}(x_{2},i)-|x_{1}-x_{2}|^{2}/2\gamma-\rho(\psi(x_{1})+\psi(x_{2})),
$$
where $0<\gamma<1$, $0<\rho<1$, and $\psi(x)=\exp(a(1+x^{2})^{\frac{1}{2}})$ with
$$
a=1\vee r^{\frac{1}{2}}\bigg(2\max_{i\in\mathcal{M}}\sup_{x\in \mathbb{R}}[|b(x,i)|+\sigma^{2}(x,i)/2]\bigg)^{-\frac{1}{2}}.
$$
It turns out that $\Psi(x_{1},x_{2},i)$ has a global maximum at a point $(x^{0}_{1},x^{0}_{2},i_{0})$,
since $\Psi(x_{1},x_{2},i)$ is continuous and $\lim_{|x_{1}|+|x_{2}|\rightarrow\infty}
\Psi(x_{1},x_{2},i)=-\infty$ for each $i\in\mathcal{M}$. In particular, we have $2\Psi(x^{0}_{1},x^{0}_{2},i_{0})\geq
\Psi(x^{0}_{1},x^{0}_{1},i_{0})+\Psi(x^{0}_{2},x^{0}_{2},i_{0})$,
i.e.,
\begin{equation}\label{quadratic}
\begin{aligned}
\frac{1}{\gamma}|x_{1}^{0}-x_{2}^{0}|^{2}
\leq [V_{1}(x_{1}^{0},i_{0})-V_{1}(x_{2}^{0},i_{0})]
+[V_{2}(x_{1}^{0},i_{0})-V_{2}(x_{2}^{0},i_{0})].
\end{aligned}
\end{equation}
It follows from the linear growth property of $V_{1}$ and $V_{2}$ that
\begin{equation}\label{convergence}
\begin{aligned}
|x_{1}^{0}-x_{2}^{0}|\leq \sqrt{K_{1}\gamma(1+|x_{1}^{0}|+|x_{2}^{0}|)},
\end{aligned}
\end{equation}
where the constant $K_{1}$ is independent of $\gamma$ and $\rho$.

Moreover, the choice of $(x_{1}^{0},x_{2}^{0},i_{0})$ implies
$\Psi(0,0,i_{0})\leq \Psi(x^{0}_{1},x^{0}_{2},i_{0})$, that is
\begin{equation*}
\begin{aligned}
\rho(\psi(x^{0}_{1})+\psi(x^{0}_{2}))
\leq& V_{1}(x^{0}_{1},i_{0})-V_{2}(x^{0}_{2},i_{0})
-\frac{1}{2\gamma}|x^{0}_{1}-x^{0}_{2}|^{2}
-\Psi(0,0,i_{0})\\
\leq& K_{2}(1+|x_{1}^{0}|+|x_{2}^{0}|),
\end{aligned}
\end{equation*}
where the constant $K_{2}$ is also independent of $\gamma$ and $\rho$.
Note that $\psi$ is of exponential growth but the right-hand side of
the above inequality only has a linear growth, hence, there should exist
a constant $K_{\rho}$, which is dependent on $\rho$ but independent of $\gamma$,
such that
\begin{equation}\label{bound}
\begin{aligned}
|x_{1}^{0}|+|x_{2}^{0}|\leq K_{\rho}.
\end{aligned}
\end{equation}
By (\ref{convergence}) and (\ref{bound}), there must be a subsequence
of $\gamma\rightarrow 0$ (still denoted by $\gamma$) and $x_{0}$ such that
$x_{1}^{0}\rightarrow x_{0}$ and $x_{2}^{0}\rightarrow x_{0}$. Further,
from the continuity of $V_{1}$ and $V_{2}$ in (\ref{quadratic}),
we deduce that
\begin{equation}\label{further deduce}
\begin{aligned}
\frac{|x_{1}^{0}-x_{2}^{0}|^{2}}{\gamma}
\rightarrow 0, \quad \gamma\rightarrow 0.
\end{aligned}
\end{equation}
Then, from Crandall et al. \cite[Theorem 3.2]{CIL1992}, there exist
$(\Lambda_{1},\Lambda_{2})\in \mathbb{R}^{2}$ such that
\begin{equation*}
\begin{aligned}
((x_{1}^{0}-x_{2}^{0})/\gamma,\Lambda_{1})
&\in\overline{J}^{2,+}(V_{1}(x^{0}_{1},i_{0})-\rho\psi(x^{0}_{1})),\\
((x_{1}^{0}-x_{2}^{0})/\gamma,\Lambda_{2})
&\in\overline{J}^{2,-}(V_{2}(x^{0}_{2},i_{0})+\rho\psi(x^{0}_{2})),
\end{aligned}
\end{equation*}
and
\begin{equation}\label{VS sigma}
\begin{aligned}
\left[
  \begin{array}{cc}
    \Lambda_{1} & 0 \\
    0 & -\Lambda_{2} \\
  \end{array}
\right]
\leq \frac{3}{\gamma}
\left[
  \begin{array}{cc}
    1 & -1 \\
    -1 &  1 \\
  \end{array}
\right].
\end{aligned}
\end{equation}
By the viscosity subsolution property of $V_{1}$ and the viscosity supersolution property
of $V_{2}$, we have
\begin{equation}\label{V1 subsolution}
\begin{aligned}
&\min\bigg\{rV_{1}(x_{1}^{0},i_{0})-b(x_{1}^{0},i_{0})[(x_{1}^{0}-x_{2}^{0})/\gamma+\rho\psi^{\prime}(x_{1}^{0})]\\
&+\frac{\theta}{2}\sigma^{2}(x_{1}^{0},i_{0})[(x_{1}^{0}-x_{2}^{0})/\gamma+\rho\psi^{\prime}(x_{1}^{0})]^{2}
-\frac{1}{2}\sigma^{2}(x_{1}^{0},i_{0})(\Lambda_{1}+\rho\psi^{\prime\prime}(x_{1}^{0}))\\
&-QV_{1}(x_{1}^{0},\cdot)(i_{0})-f(x_{1}^{0},i_{0}),V_{1}(x_{1}^{0},i_{0})-g(x_{1}^{0},i_{0})\bigg\}\leq0,
\end{aligned}
\end{equation}
and
\begin{equation}\label{V2 supersolution}
\begin{aligned}
&\min\bigg\{rV_{2}(x_{2}^{0},i_{0})-b(x_{2}^{0},i_{0})[(x_{1}^{0}-x_{2}^{0})/\gamma-\rho\psi^{\prime}(x_{2}^{0})]\\
&+\frac{\theta}{2}\sigma^{2}(x_{2}^{0},i_{0})[(x_{1}^{0}-x_{2}^{0})/\gamma-\rho\psi^{\prime}(x_{2}^{0})]^{2}
-\frac{1}{2}\sigma^{2}(x_{2}^{0},i_{0})(\Lambda_{2}-\rho\psi^{\prime\prime}(x_{2}^{0}))\\
&-QV_{2}(x_{2}^{0},\cdot)(i_{0})-f(x_{2}^{0},i_{0}),V_{2}(x_{2}^{0},i_{0})-g(x_{2}^{0},i_{0})\bigg\}\geq0.
\end{aligned}
\end{equation}
In the following, based on (\ref{V1 subsolution}), the analysis is divided into two cases.

\emph{Case 1.} If $V_{1}(x_{1}^{0},i_{0})-g(x_{1}^{0},i_{0})\leq0$,
then from $V_{2}(x_{2}^{0},i_{0})-g(x_{2}^{0},i_{0})\geq0$, we get
$V_{1}(x_{1}^{0},i_{0})-V_{2}(x_{2}^{0},i_{0})\leq g(x_{1}^{0},i_{0})-g(x_{2}^{0},i_{0})$.
Letting $\gamma\rightarrow0$ yields $V_{1}(x_{0},i_{0})-V_{2}(x_{0},i_{0})\leq 0$.

\emph{Case 2.} If $V_{1}(x_{1}^{0},i_{0})-g(x_{1}^{0},i_{0})>0$, then the left-hand side
of $\min\{\ldots,\ldots\}$ in (\ref{V1 subsolution}) should be $\leq0$. By subtracting
the left-hand side of $\min\{\ldots,\ldots\}$ in (\ref{V2 supersolution}),
and sending $\gamma\rightarrow 0$, it follows from (A1), (A2), (\ref{further deduce}), (\ref{VS sigma}) that
\begin{equation*}
\begin{aligned}
&V_{1}(x_{0},i_{0})-V_{2}(x_{0},i_{0})\\
\leq&\frac{2\rho}{r}\bigg(b(x_{0},i_{0})\psi^{\prime}(x_{0})
+\frac{1}{2}\sigma^{2}(x_{0},i_{0})\psi^{\prime\prime}(x_{0})\bigg)
+\frac{1}{r}[QV_{1}(x_{0},\cdot)(i_{0})-QV_{2}(x_{0},\cdot)(i_{0})].
\end{aligned}
\end{equation*}
On the other hand, since $\Psi(x_{1},x_{2},i)$ reaches its maximum at $(x_{1}^{0},x_{2}^{0},i_{0})$,
it follows that for any $(x,i)\in \mathbb{R}\times \mathcal{M}$,
\begin{equation*}
\begin{aligned}
V_{1}(x,i)-V_{2}(x,i)-2\rho\psi(x)=&\Psi(x,x,i)\\
\leq&\Psi(x^{0}_{1},x^{0}_{2},i_{0})
\leq V_{1}(x^{0}_{1},i_{0})-V_{2}(x^{0}_{2},i_{0})-\rho(\psi(x^{0}_{1})+\psi(x^{0}_{2})).
\end{aligned}
\end{equation*}
Letting $\gamma\rightarrow 0$, we have
\begin{equation*}
\begin{aligned}
V_{1}(x,i)-V_{2}(x,i)-2\rho\psi(x)
\leq V_{1}(x_{0},i_{0})-V_{2}(x_{0},i_{0})-2\rho\psi(x_{0}).
\end{aligned}
\end{equation*}
Particularly, taking $x=x_{0}$ in the above equation also leads to
$QV_{1}(x_{0},\cdot)(i_{0})-QV_{2}(x_{0},\cdot)(i_{0})\leq0$.
So no matter in Case 1 or Case 2, we both have $V_{1}(x,i)-V_{2}(x,i)\leq2\rho\psi(x)$.
Sending $\rho\rightarrow 0$, we have $V_{1}(x,i)-V_{2}(x,i)\leq 0$.
Using a similar argument, we can also obtain $V_{1}(x,i)-V_{2}(x,i)\geq 0$.
Thus we conclude that $V_{1}(x,i)=V_{2}(x,i)$, $\forall (x,i)\in \mathbb{R}\times \mathcal{M}$.
~\hfill $\Box$

\section{Smooth-fit principle}\label{Section SFP}

In this section, we prove the smooth-fit $C^{1}$ property of the value function $V(x,i)$
through the boundary $\partial\mathcal{D}_{i}$ between the continuation region
$\mathcal{D}_{i}\doteq\{x\in \mathbb{R}:V(x,i)>g(x,i)\}$ and its complement set, i.e.,
the stopping region $\mathcal{D}_{i}^{c}\doteq\{x\in \mathbb{R}:V(x,i)=g(x,i)\}$ via
a viscosity solution approach. In this section, we continue to adopt $\Theta(q)=q^{2}/2\theta$,
or to say, we shall work with the HJB equation (\ref{uniqueness HJB}), as the uniqueness
of viscosity solution is needed. For convenience, we denote the left-hand side of
$\min\{\ldots,\ldots\}$ in (\ref{uniqueness HJB})
as $H^{*}v(x,i)=rv-bv^{\prime}+\frac{\theta}{2}\sigma^{2}[v^{\prime}]^{2}
-\frac{1}{2}\sigma^{2}v^{\prime\prime}-Qv(x,\cdot)(i)-f$.
\begin{lemma}\label{Lemma Classical}
Let (A1)-(A3) hold. Assume further that: {\rm(\romannumeral1)} $g(x,i)\in C^{2}$,
and {\rm(\romannumeral2)} $\sigma(x,i)$ is non-degenerate, i.e., $\sigma^{2}(x,i)>0$.
Then, the value function $V(x,i)$ is the unique classical solution to $H^{*}v(x,i)=0$
on $\mathcal{D}_{i}$.
\end{lemma}
\noindent{\it Proof.}
\textbf{Step 1}. We first show that $V(x,i)$ is the unique viscosity solution to $H^{*}v(x,i)=0$
on $\mathcal{D}_{i}$. In fact, the viscosity supersolution property of $V(x,i)$ to $H^{*}v(x,i)=0$
is immediate from that of $V(x,i)$ to (\ref{uniqueness HJB}). On the other hand, for any $i\in\mathcal{M}$,
let $\overline{x}\in \mathcal{D}_{i}$ and $\varphi(x)\in C^{2}$ be such that $V(x,i)-\varphi(x)$
attains its local maximum at $x=\overline{x}$. By the definition of $\mathcal{D}_{i}$, we have
$V(\overline{x},i)>g(\overline{x},i)$, hence from the viscosity subsolution property of $V(x,i)$
to (\ref{uniqueness HJB}), we have
\begin{equation*}
\begin{aligned}
rV(\overline{x},i)-b(\overline{x},i)\varphi^{\prime}(\overline{x},i)
+\frac{\theta}{2}\sigma^{2}(\overline{x},i)[\varphi^{\prime}(\overline{x},i)]^{2}
-\frac{1}{2}\sigma^{2}(\overline{x},i)\varphi^{\prime\prime}(\overline{x},i)&\\
-QV(\overline{x},\cdot)(i)-f(\overline{x},i)&\leq0,
\end{aligned}
\end{equation*}
which leads to the viscosity subsolution property of $V(x,i)$ to $H^{*}v(x,i)=0$. Moreover,
similar to Theorem \ref{uniqueness theorem}, we can prove the uniqueness of viscosity solution
to $H^{*}v(x,i)=0$ on $\mathcal{D}_{i}$ (In the case that $\mathcal{D}_{i}$ is a bounded domain,
we even no longer need the function $\psi$ in $\Psi$, which is introduced to deal with
an unbounded domain by imposing restrictions on the growth of $\Psi$, in the proof of
Theorem \ref{uniqueness theorem}).

\textbf{Step 2}. We rewrite the equation $H^{*}v(x,i)=0$ as
\begin{equation*}
\begin{aligned}
v^{\prime\prime}(x,i)
=\frac{2r}{\sigma^{2}(x,i)}v(x,i)-\frac{2b(x,i)}{\sigma^{2}(x,i)}v^{\prime}(x,i)+\theta[v^{\prime}(x,i)]^{2}
-\sum_{j=1}^{m}\frac{2\lambda_{ij}}{\sigma^{2}(x,i)}v(x,j)-\frac{2f(x,i)}{\sigma^{2}(x,i)}.
\end{aligned}
\end{equation*}
Let
$$
\mathbf{y}(x)\doteq(v(x,1),\ldots,v(x,m),v^{\prime}(x,1),\ldots,v^{\prime}(x,m))^{\top},
$$
where the superscript $\top$ denotes the transpose of a vector. Then, the above equation
can be further rewritten as a matrix form:
$$
\mathbf{y}^{\prime}(x)=\mathbf{A}(x,\mathbf{y}(x))\mathbf{y}(x)+\mathbf{F}(x),
$$
where
$$
\mathbf{F}(x)=\left(0,\ldots,0,\frac{-2f(x,1)}{\sigma^{2}(x,1)},\ldots,\frac{-2f(x,m)}{\sigma^{2}(x,m)}\right)^{\top},
$$
and
\begin{equation*}
\begin{aligned}
\mathbf{A}(x,\mathbf{y})=\left[
         \begin{array}{cc}
           \mathbf{O}_{m} & \mathbf{E}_{m} \\
           \mathbf{A}_{21}(x) & \mathbf{A}_{22}(x,\mathbf{y}) \\
         \end{array}
       \right]_{2m\times2m},
\end{aligned}
\end{equation*}
in which $\mathbf{O}_{m}$ is the $(m\times m)$ zero matrix, $\mathbf{E}_{m}$ is the $(m\times m)$
identity matrix,
\begin{equation*}
\begin{aligned}
\mathbf{A}_{21}(x)
=\left[
   \begin{array}{cccc}
     \frac{2(r-\lambda_{11})}{\sigma^{2}(x,1)} & \frac{-2\lambda_{12}}{\sigma^{2}(x,1)} & \cdots & \frac{-2\lambda_{1m}}{\sigma^{2}(x,1)} \\
     \frac{-2\lambda_{21}}{\sigma^{2}(x,2)} & \frac{2(r-\lambda_{22})}{\sigma^{2}(x,2)} & \cdots & \frac{-2\lambda_{2m}}{\sigma^{2}(x,2)} \\
     \vdots & \vdots &  & \vdots \\
     \frac{-2\lambda_{m1}}{\sigma^{2}(x,m)} & \frac{-2\lambda_{m2}}{\sigma^{2}(x,m)} & \cdots & \frac{2(r-\lambda_{mm})}{\sigma^{2}(x,m)} \\
   \end{array}
 \right]_{m\times m},
\end{aligned}
\end{equation*}
and
\begin{equation*}
\begin{aligned}
\mathbf{A}_{22}(x,\mathbf{y})
=\left[
   \begin{array}{cccc}
     \frac{-2b(x,1)}{\sigma^{2}(x,1)}+\theta v^{\prime}(\cdot,1) & 0 & \cdots & 0 \\
     0 & \frac{-2b(x,1)}{\sigma^{2}(x,2)}+\theta v^{\prime}(\cdot,2) & \cdots & 0 \\
     \vdots  &  \vdots & \ddots  & \vdots \\
     0 & 0 & \cdots & \frac{-2b(x,m)}{\sigma^{2}(x,m)}+\theta v^{\prime}(\cdot,m) \\
   \end{array}
 \right]_{m\times m}.
\end{aligned}
\end{equation*}

\textbf{Step 3}. Note that the value functions $V(x,i)$, $i\in\mathcal{M}$, are Lipschitz continuous,
thus they are differentiable almost everywhere. Denote $L$ the common Lipschitz constant for all $V(x,i)$.
Denote $E_{i}$ the set of points at which $V(x,i)$ is non-differentiable, then $E_{i}$ is a zero measure
set. Moreover, $E\doteq \cup_{i\in\mathcal{M}}E_{i}$ is also a zero measure set as $\mathcal{M}$ is
a finite set. Now, for any $i\in\mathcal{M}$ and any $x_{0}\in\mathcal{D}_{i}$, consider an interval
$[x_{0}-l,x_{0}+l]\subset\mathcal{D}_{i}$ for some $l>0$. Note that $V(x,i)$, $i\in\mathcal{M}$, are
bounded in $[x_{0}-l,x_{0}+l]$; without loss of generality, we denote the common bound also by $L$,
i.e., $|V(x,i)|\leq L$, $i\in\mathcal{M}$. For a given $b>0$, let $M$ be the maximum of
$|\mathbf{A}(x,\mathbf{y})\mathbf{y}+\mathbf{F}(x)|$ in the region
\begin{equation*}
\begin{aligned}
I_{0}:\quad x\in[x_{0}-l,x_{0}+l],\quad \mathbf{y}\in [-L-b,L+b]^{2m}.
\end{aligned}
\end{equation*}
We do not know in advance if all $V^{\prime}(x_{0},i)$, $i\in\mathcal{M}$, exist, but we can indeed,
in view of that $E$ is a zero measure set, select an $x_{1}$ near $x_{0}$ (let $|x_{1}-x_{0}|=\varepsilon$
for some sufficiently small $\varepsilon<\min\{l/2,b/M\}$) such that all $V^{\prime}(x_{1},i)$, $i\in\mathcal{M}$,
exist. Then, let us consider the following initial value problem:
\begin{equation}\label{initial value problem}
\left\{
\begin{aligned}
\mathbf{y}^{\prime}(x)=&\mathbf{A}(x,\mathbf{y}(x))\mathbf{y}(x)+\mathbf{F}(x),\\
\mathbf{y}(x_{1})=&\mathbf{y}_{1},
\end{aligned}
\right.
\end{equation}
where
$$
\mathbf{y}_{1}\doteq(V(x_{1},1),\ldots,V(x_{1},m),V^{\prime}(x_{1},1),\ldots,V^{\prime}(x_{1},m))^{\top},
$$
in the following region
\begin{equation*}
\begin{aligned}
I_{1}:\quad |x-x_{1}|\leq a,\quad |\mathbf{y}-\mathbf{y}_{1}|\leq b,
\end{aligned}
\end{equation*}
for some $a$ satisfying $\varepsilon<a<l-\varepsilon$; it means $I_{1}\subset I_{0}$. The function
$\mathbf{A}(x,\mathbf{y})\mathbf{y}+\mathbf{F}(x)$ is continuous w.r.t. $(x,\mathbf{y})$ and
Lipschitz continuous w.r.t. $\mathbf{y}$ on the region $I_{1}$. From the Picard-Lindel\"{o}f
theorem (see Hartman \cite[Chapter \uppercase\expandafter{\romannumeral2}, Theorem 1.1]{Hartman2002}),
the initial value problem (\ref{initial value problem}) has a unique classical solution $\mathbf{y}(x)$
in $[x_{1}-h,x_{1}+h]$, where (recall that $M$ is an upper bound of
$|\mathbf{A}(x,\mathbf{y})\mathbf{y}+\mathbf{F}(x)|$ on $I_{1}\subset I_{0}$)
\begin{equation*}
\begin{aligned}
h\doteq\min\bigg\{a,\frac{b}{M}\bigg\}>\varepsilon.
\end{aligned}
\end{equation*}
In particular, the equation $H^{*}v(x,i)=0$ with initial value
$v(x_{1},i)=V(x_{1},i)$ and $v^{\prime}(x_{1},i)=V^{\prime}(x_{1},i)$ has a unique classical
solution $v(x,i)$ in $[x_{1}-h,x_{1}+h]$. Such a classical solution is a viscosity solution as well.

Note that $[x_{1}-h,x_{1}+h]\subset \mathcal{D}_{i}$.
Remember that $V(x,i)$ is a viscosity solution to $H^{*}v(x,i)=0$ on $\mathcal{D}_{i}$, especially,
in $(x_{1}-h,x_{1}+h)$. From the uniqueness of viscosity solution, we have $v(x,i)\equiv V(x,i)$
in $(x_{1}-h,x_{1}+h)$. So $V(x,i)$ is smooth $C^{2}$ in $(x_{1}-h,x_{1}+h)$. It yields that
$V^{\prime}(x_{0},i)$ indeed exists since $x_{0}\in(x_{1}-h,x_{1}+h)$. Moreover,
if $x_{0}\in\mathcal{D}_{j}$ for $j\neq i$, we can repeat the above proof to show
the well-posedness of $V^{\prime}(x_{0},j)$; otherwise, if $x_{0}\notin\mathcal{D}_{j}$,
then $x_{0}\in\mathcal{D}_{j}^{c}$ and we have $V(x,j)\equiv g(x,j)$, which implies
that $V(x,j)$ is differentiable at $x_{0}$ with $V^{\prime}(x_{0},j)=g^{\prime}(x_{0},j)$.
Finally, we show that all $V^{\prime}(x_{0},i)$, $i\in\mathcal{M}$, exist for any point
$x_{0}\in\mathcal{D}_{i}$.

\textbf{Step 4}. Now, let us consider the initial value problem (\ref{initial value problem})
with new initial condition $\mathbf{y}(x_{0})=\mathbf{y}_{0}$, where
$$
\mathbf{y}_{0}\doteq(V(x_{0},1),\ldots,V(x_{0},m),V^{\prime}(x_{0},1),\ldots,V^{\prime}(x_{0},m))^{\top}.
$$
As in Step 3, we can find a solution interval $(x_{0}-h,x_{0}+h)$ in which we have $v(x,i)\equiv V(x,i)$,
and thus $V(x,i)$ is smooth $C^{2}$ in $(x_{0}-h,x_{0}+h)$. From the arbitrariness of $x_{0}\in\mathcal{D}_{i}$,
we have $V(x,i)$ is smooth $C^{2}$ on $\mathcal{D}_{i}$. It is known that a smooth enough viscosity solution
is a classical solution, hence $V(x,i)$ satisfies $H^{*}v(x,i)=0$ on $\mathcal{D}_{i}$ in the classical sense.
~\hfill $\Box$
\begin{theorem}\label{Theorem SFP}
Let (A1)-(A3) hold. Assume further that: {\rm(\romannumeral1)} $g(x,i)\in C^{2}$,
and {\rm(\romannumeral2)} $\sigma(x,i)$ is non-degenerate, i.e., $\sigma^{2}(x,i)>0$.
Then, the value function $V(x,i)$ is smooth $C^{1}$ on $\partial\mathcal{D}_{i}$.
\end{theorem}
\noindent{\it Proof.}
We first verify that, for any $\overline{x}\in\partial\mathcal{D}_{i}$,
$V(x,i)$ admits a left derivative $V^{\prime}_{-}(\overline{x},i)$ and a right
derivative $V^{\prime}_{+}(\overline{x},i)$. Without loss of generality,
we suppose that $\overline{x}$ is a right boundary of $\mathcal{D}_{i}$, that is,
the left-side neighborhood $(\overline{x}-\delta,\overline{x})\subset\mathcal{D}_{i}$
and the right-side neighborhood $(\overline{x},\overline{x}+\delta)\subset\mathcal{D}_{i}^{c}$
for some $\delta>0$; the other case that $\overline{x}$ is a left boundary of $\mathcal{D}_{i}$
can be treated similarly. In $(\overline{x}-\delta,\overline{x})$, we have proved
in Lemma \ref{Lemma Classical} that $V(x,i)=v(x,i)$, the unique classical solution
to $H^{*}v(x,i)=0$ on $\mathcal{D}_{i}$, thus we have $V^{\prime}_{-}(\overline{x},i)=v^{\prime}(\overline{x},i)$.
On the other hand, in $(\overline{x},\overline{x}+\delta)$, it follows from $V=g$ on $\mathcal{D}_{i}^{c}$ that
$V^{\prime}_{+}(\overline{x},i)=g^{\prime}(\overline{x},i)$.

In view of $V(\overline{x},i)=g(\overline{x},i)$ and $V(x,i)\geq g(x,i)$ for any
$(x,i)\in \mathbb{R}\times \mathcal{M}$, we obtain the following relation:
\begin{equation*}
\begin{aligned}
\frac{V(x,i)-V(\overline{x},i)}{x-\overline{x}}
\leq\frac{g(x,i)-g(\overline{x},i)}{x-\overline{x}},\quad x<\overline{x},
\end{aligned}
\end{equation*}
which means that $V^{\prime}_{-}(\overline{x},i)\leq g^{\prime}(\overline{x},i)=V^{\prime}_{+}(\overline{x},i)$.
Suppose, on the contrary, $V$ is not $C^{1}$ at $\overline{x}$,
then we have $V^{\prime}_{-}(\overline{x},i)<V^{\prime}_{+}(\overline{x},i)$.
Taking an $\eta\in (V^{\prime}_{-}(\overline{x},i),V^{\prime}_{+}(\overline{x},i))$,
we consider the following smooth test function:
\begin{equation*}
\begin{aligned}
\varphi_{\mu}(x)=V(\overline{x},i)+\eta(x-\overline{x})+\frac{1}{2\mu}(x-\overline{x})^{2}.
\end{aligned}
\end{equation*}
It can be seen that $V$ dominates locally $\varphi_{\mu}$ in a neighborhood of $\overline{x}$,
i.e., $V-\varphi_{\mu}$ achieves a local minimum at $\overline{x}$. Hence, from the
viscosity supersolution property of $V$ to the HJB equation (\ref{uniqueness HJB}), we have
\begin{equation*}
\begin{aligned}
rV(\overline{x},i)-b(\overline{x},i)\eta+\frac{\theta}{2}\sigma^{2}(\overline{x},i)\eta^{2}
-\frac{1}{2\mu}\sigma^{2}(\overline{x},i)-QV(\overline{x},\cdot)(i)-f(\overline{x},i)\geq0,
\end{aligned}
\end{equation*}
which, by sending $\mu\rightarrow0$ and noting the non-degenerate condition $\sigma^{2}>0$,
yields a contradiction.
~\hfill $\Box$

\section{Verification theorem}\label{Section VT}

Next, we present the verification theorem, which gives a sufficient way to determine whether
a candidate is optimal. It also suggests that the solution to the HJB equation (if exists)
coincides with the value function. In what follows, we denote
$H^{q}v(x,i)=rv(x,i)-\mathcal{L}^{q}v(x,i)-Qv(x,\cdot)(i)-f(x,i)-\Theta(q)$.
\begin{theorem}\label{VT}
Let $v(x,i)$ be a real-valued function satisfying the following conditions:
{\rm(\romannumeral1)} $v(\cdot,i)\in C^{2}(\mathbb{R}\backslash\partial \mathcal{D}_{i})\cap C^{1}(\mathbb{R})$,
{\rm(\romannumeral2)} $v(x,i)\geq g(x,i)$,
{\rm(\romannumeral3)} $\sup_{q\in \mathbb{R}}H^{q}v(x,i)\geq0$.
Then, we have $v(x,i)\geq V(x,i)$ for all $(x,i)\in \mathbb{R}\times \mathcal{M}$.

Assume further that: {\rm(\romannumeral4)} For $x\in \mathcal{D}_{i}$, $\sup_{q\in \mathbb{R}}H^{q}v(x,i)=0$,
and there exists a feedback map $q^{*}(x,i)$ such that
\begin{equation}\label{q*}
\begin{aligned}
q^{*}(x,i)=\arg\sup_{q\in \mathbb{R}}H^{q}v(x,i).
\end{aligned}
\end{equation}
Let $q_{t}^{*}\doteq q^{*}(X_{t}^{q_{t}^{*}},\alpha_t)\in\mathcal{U}$. Then, we have
\begin{equation}\label{optimal stopping rule}
\begin{aligned}
\tau^{*}\doteq\inf\{t\geq0:X_{t}^{q_{t}^{*}}\notin \mathcal{D}_{\alpha_{t}}\}
\end{aligned}
\end{equation}
is a robust optimal stopping time and $v(x,i)=V(x,i)$ is the value function.
\end{theorem}
\begin{remark}
It should be noted that along the boundary $\partial \mathcal{D}_{i}$, the functions $v(\cdot,i)$
only belongs to $C^{1}$ but does not necessarily belong to $C^{2}$. In this situation, we can take
advantage of the smooth approximation argument for variational inequalities developed
by {\O}ksendal \cite[Theorem 10.4.1 and Appendix D]{Oksendal2003} to supply the needed
smoothness when applying It\^{o}'s formula. Therefore, for convenience, in the following proof
we simply consider $v(\cdot,i)$ to be $C^{2}$ on the whole space; see also Guo and Zhang \cite{GZ2005}.
\end{remark}
\noindent{\it Proof.}
\textbf{Step 1.} For any $\tau\in\mathcal{S}$ and $q_{t}\in\mathcal{U}$,
by applying It\^{o}'s formula to $e^{-rt}v(X_{t}^{q_{t}},\alpha_{t})$
between 0 and $\tau$, we have
\begin{equation}\label{VT proof 1}
\begin{aligned}
&\mathbb{E}^{\mathbb{Q}}[e^{-r\tau}v(X_{\tau}^{q_{t}},\alpha_{\tau})]-v(x,i)\\
=&\mathbb{E}^{\mathbb{Q}}\bigg[\int_{0}^{\tau}e^{-rt}\bigg(-rv(X_{t}^{q_{t}},\alpha_{t})
+\mathcal{L}^{q_{t}}v(X_{t}^{q_{t}},\alpha_{t})+Qv(X_{t}^{q_{t}},\cdot)(\alpha_{t})\bigg)dt\bigg]\\
\leq&\mathbb{E}^{\mathbb{Q}}\bigg[\int_{0}^{\tau}e^{-rt}
\bigg(\mathcal{L}^{q_{t}}v(X_{t}^{q_{t}},\alpha_{t})-f(X_{t}^{q_{t}},\alpha_{t})
-\inf_{q\in \mathbb{R}}\bigg\{\mathcal{L}^{q}v(X_{t}^{q},\alpha_{t})+\Theta(q)\bigg\}\bigg)dt\bigg],
\end{aligned}
\end{equation}
where the inequality follows from condition (\romannumeral3).
In view of condition (\romannumeral2), we obtain
\begin{equation}\label{VT proof 2}
\begin{aligned}
J(x,i;\tau,q_{t})
\leq\mathbb{E}^{\mathbb{Q}}\bigg[e^{-r\tau}v(X_{\tau}^{q_{t}},\alpha_{\tau})
+\int_{0}^{\tau}e^{-rt}\bigg(f(X_{t}^{q_{t}},\alpha_{t})+\Theta(q_{t})\bigg)dt\bigg].
\end{aligned}
\end{equation}
From (\ref{VT proof 1}) and (\ref{VT proof 2}), we have
\begin{equation}\label{VT proof 6}
\begin{aligned}
J(x,i;\tau,q_{t})
\leq v(x,i)+\mathbb{E}^{\mathbb{Q}}\bigg[\int_{0}^{\tau}e^{-rt}
\bigg(\Pi^{q_{t}}_{t}-\inf_{q\in \mathbb{R}}\Pi^{q}_{t}\bigg)dt\bigg],
\end{aligned}
\end{equation}
where we denote $\Pi^{q}_{t}=\mathcal{L}^{q}v(X_{t}^{q},\alpha_{t})+\Theta(q)$.
By taking infimum over $q_{t}\in\mathcal{U}$ in (\ref{VT proof 6}), we have
\begin{equation}\label{VT proof 3}
\begin{aligned}
\widetilde{J}(x,i;\tau)\doteq&\inf_{q_{t}\in\mathcal{U}}J(x,i;\tau,q_{t})\\
\leq& v(x,i)+\inf_{q_{t}\in\mathcal{U}}
\mathbb{E}^{\mathbb{Q}}\bigg[\int_{0}^{\tau}e^{-rt}
\bigg(\Pi^{q_{t}}_{t}-\inf_{q\in \mathbb{R}}\Pi^{q}_{t}\bigg)dt\bigg]
=v(x,i),
\end{aligned}
\end{equation}
i.e., $\widetilde{J}(x,i;\tau)\leq v(x,i)$.
Further, from the arbitrariness of $\tau\in\mathcal{S}$, it follows that
\begin{equation}\label{VT proof 4}
\begin{aligned}
V(x,i)=\sup_{\tau\in\mathcal{S}}\widetilde{J}(x,i;\tau)\leq v(x,i).
\end{aligned}
\end{equation}
\textbf{Step 2.} In this step, we take the feedback control $q_{t}^{*}$ defined
by (\ref{q*}) and the stopping time $\tau^{*}$ defined by (\ref{optimal stopping rule}),
then condition (\romannumeral4) implies that the inequality in (\ref{VT proof 1})
becomes equality. The definition (\ref{optimal stopping rule}) of $\tau^{*}$ makes
the inequality in (\ref{VT proof 2}) become equality and so does the inequality
in (\ref{VT proof 6}). Then, (\ref{VT proof 3}) also turns out to be an equality,
and $q_{t}^{*}$ defined by (\ref{q*}) indeed achieves the infimum of $\inf_{q_{t}\in\mathcal{U}}$
in (\ref{VT proof 3}). In particular,
\begin{equation}\label{VT proof 5}
\begin{aligned}
\widetilde{J}(x,i;\tau^{*})=v(x,i).
\end{aligned}
\end{equation}
It follows from (\ref{VT proof 4}) and (\ref{VT proof 5}) that $V(x,i)=\widetilde{J}(x,i;\tau^{*})=v(x,i)$,
i.e., $\tau^{*}$ defined by (\ref{optimal stopping rule}) is a robust optimal stopping time
and $v(x,i)=V(x,i)$ is the value function.
~\hfill $\Box$

\section{Two-time-scale case}\label{Section TTS}

In this section, we suppose that the state space of a Markov chain $\alpha^{\varepsilon}_{t}$,
where $\varepsilon$ is the so-called \emph{time-scale parameter}, is divided into a number of
groups such that the chain fluctuates rapidly within a group, but jumps occasionally from one
group to another. Let the generator $Q^{\varepsilon}\doteq(\lambda^{\varepsilon}_{ij})_{i,j\in\mathcal{M}}$
for $\alpha^{\varepsilon}_{t}$ have the following two-time-scale structure:
\begin{equation*}
\begin{aligned}
Q^{\varepsilon}
=\frac{1}{\varepsilon}\widetilde{Q}+\widehat{Q}
\doteq\frac{1}{\varepsilon}
{\rm diag}\{\widetilde{Q}^{1},\ldots,\widetilde{Q}^{L}\}+\widehat{Q},
\end{aligned}
\end{equation*}
where each $\widetilde{Q}^{k}$ is an irreducible generator with dimension $m_{k}\times m_{k}$
and $m_{1}+\cdots+m_{L}=m$. In this case, we can write the state space of $\alpha^{\varepsilon}_{t}$
as $\mathcal{M}=\mathcal{M}_{1}\cup\cdots\cup\mathcal{M}_{L}
=\{s_{11},\ldots,s_{1m_{1}}\}\cup\cdots\cup\{s_{L1},\ldots,s_{Lm_{L}}\}$,
where, for $k=1,\ldots,L$, $\mathcal{M}_{k}\doteq\{s_{k1},\ldots,s_{km_{k}}\}$ is the subspace
corresponding to $\widetilde{Q}^{k}$. Let $\nu^{k}=(\nu^{k}_{1},\ldots,\nu^{k}_{m_{k}})$ denote
the \emph{stationary distribution} of $\widetilde{Q}^{k}$.
Note that in $Q^{\varepsilon}$, $\widetilde{Q}\doteq(\widetilde{\lambda}_{ij})_{i,j\in\mathcal{M}}$
governs the rapidly changing part and $\widehat{Q}\doteq(\widehat{\lambda}_{ij})_{i,j\in\mathcal{M}}$
describes the slowly varying components.

Define a process $\overline{\alpha}^{\varepsilon}_{t}$ by $\overline{\alpha}^{\varepsilon}_{t}=k$
if $\alpha^{\varepsilon}_{t}\in \mathcal{M}_{k}$. The process $\overline{\alpha}^{\varepsilon}_{t}$
is not necessarily a Markov chain. However, using a certain probabilistic argument, it is shown in
Yin and Zhang \cite[Chapter 4]{YinZhang2013} that $\overline{\alpha}^{\varepsilon}_{t}$ converges
in distribution to a \emph{limit Markov chain} $\overline{\alpha}_{t}$ with state space $\{1,\ldots,L\}$,
whose generator $\overline{Q}\doteq(\overline{\lambda}_{kp})_{k,p\in\{1,\ldots,L\}}$ is given by
\begin{equation}\label{limit generator}
\begin{aligned}
\overline{Q}
={\rm diag}\{\nu^{1},\ldots,\nu^{L}\}\widehat{Q}{\rm diag}\{1_{m_{1}},\ldots,1_{m_{L}}\},
\end{aligned}
\end{equation}
where $1_{m_{k}}\doteq (1,\ldots,1)^{\top}\in \mathbb{R}^{m_{k}}$. The rationale of the above procedure
is that, when $\varepsilon$ is small, one may ignore some details of the Markov chain $\alpha^{\varepsilon}_{t}$
and obtain a limit problem in which $\alpha^{\varepsilon}_{t}$ is \emph{aggregated} so that the states
in $\mathcal{M}_{k}$ can be regarded as a single ``superstate". The structure of the limit problem is
much simpler than the original one, and the solution of the limit problem converges to that of the original problem.

In this section, we also take $\Theta(q)$ to be the relative entropy form $\Theta(q)=q^{2}/2\theta$,
as the uniqueness of viscosity solution is still needed, and we further assume:

(A4) $g(x,i)\equiv g(x)$, $i\in\mathcal{M}$.

For $k=1,\ldots,L$, we define a set of \emph{averaged coefficients} weighted by
the stationary distributions ($\varphi\doteq b,f$):
\begin{equation}\label{averaged coefficients}
\begin{aligned}
\overline{\varphi}(x,k)=\sum_{r=1}^{m_{k}}\nu^{k}_{r}\varphi(x,s_{kr}),\quad
\overline{\sigma}^{2}(x,k)=\sum_{r=1}^{m_{k}}\nu^{k}_{r}\sigma^{2}(x,s_{kr}).
\end{aligned}
\end{equation}
Now, we consider a \emph{limit optimal stopping problem} with the averaged coefficients
$\overline{b}$, $\overline{\sigma}$, $\overline{f}$ given by (\ref{averaged coefficients})
and the limit Markov chain $\overline{\alpha}_{t}$ defined by (\ref{limit generator}).
Denote $V^{\varepsilon}(x,i)$ and $V^{0}(x,k)$ the value functions of the original problem
and limit problem, respectively. Then, $V^{\varepsilon}(x,i)$ is the unique viscosity solution
to the HJB equation (\ref{uniqueness HJB}), and $V^{0}(x,k)$ is the unique viscosity solution
to the following \emph{limit HJB equation} (the arguments $(x,k)$ are omitted):
\begin{equation}\label{limit HJB}
\begin{aligned}
\min\bigg\{rv-\overline{b}v^{\prime}+\frac{\theta}{2}\overline{\sigma}^{2}[v^{\prime}]^{2}
-\frac{1}{2}\overline{\sigma}^{2}v^{\prime\prime}
-\overline{Q}v(x,\cdot)(k)-\overline{f},
v-g\bigg\}=0.
\end{aligned}
\end{equation}
Note that the number of equations for $V^{\varepsilon}$ is $m=m_{1}+\cdots+m_{L}$, while that
for $V^{0}$ is only $L$. From the Lipschitz property of the value function and
the Arzel\`{a}-Ascoli theorem, on any compact subset of $\mathbb{R}$, for each subsequence
of $\{\varepsilon\rightarrow0\}$, there exists a further subsequence such that $V^{\varepsilon}$
converges to a limit function $\overline{V}$ on that set. The following lemma indicates that,
$\overline{V}$ depends only on $k$ whenever $\alpha^{\varepsilon}_{t}\in \mathcal{M}_{k}$
for $k=1,\ldots,L$.
\begin{lemma}\label{Independent}
Let (A1)-(A4) hold. Then, $\overline{V}(x,s_{kl})$ is independent of $l$,
where $s_{kl}\in\mathcal{M}_{k}=\{s_{k1},\ldots,s_{km_{k}}\}$.
\end{lemma}
\begin{proof}
For $(\eta,\Lambda)\in J^{2,-}\overline{V}(x,s_{kl})$, from Crandall et al. \cite[Lemma 6.1]{CIL1992},
there exist $\varepsilon_{n}\rightarrow0$, $x_{n}\rightarrow x$,
and $(\eta_{n},\Lambda_{n})\in J^{2,-}V^{\varepsilon_{n}}(x_{n},s_{kl})$
such that $(V^{\varepsilon_{n}}(x_{n},s_{kl}),\eta_{n},\Lambda_{n})
\rightarrow(\overline{V}(x,s_{kl}),\eta,\Lambda)$, as $n\rightarrow\infty$.
Further, the definition of viscosity supersolution yields
\begin{equation*}
\begin{aligned}
&rV^{\varepsilon_{n}}(x_{n},s_{kl})-b(x_{n},s_{kl})\eta_{n}+\frac{\theta}{2}\sigma^{2}(x_{n},s_{kl})\eta_{n}^{2}
-\frac{1}{2}\sigma^{2}(x_{n},s_{kl})\Lambda_{n}\\
&-\sum_{j\neq s_{kl}}\lambda^{\varepsilon_{n}}_{s_{kl},j}
[V^{\varepsilon_{n}}(x_{n},j)-V^{\varepsilon_{n}}(x_{n},s_{kl})]-f(x_{n},s_{kl})\geq0.
\end{aligned}
\end{equation*}
Recall that
$\lambda^{\varepsilon_{n}}_{s_{kl},j}
=\frac{1}{\varepsilon_{n}}\widetilde{\lambda}_{s_{kl},j}+\widehat{\lambda}_{s_{kl},j}$
and $\widetilde{\lambda}_{s_{kl},j}=0$ when $j\notin \mathcal{M}_{k}$.
Multiplying both sides of the above inequality by $\varepsilon_{n}$
and letting $n\rightarrow\infty$, it follows that
$
\sum_{r\neq l}\widetilde{\lambda}_{s_{kl},s_{kr}}
[\overline{V}(x,s_{kr})-\overline{V}(x,s_{kl})]\leq0,
$
or,
\begin{equation*}
\begin{aligned}
\overline{V}(x,s_{kl})\geq\sum_{r\neq l}
\bigg(-\frac{\widetilde{\lambda}_{s_{kl},s_{kr}}}{\widetilde{\lambda}_{s_{kl},s_{kl}}}\bigg)
\overline{V}(x,s_{kr}).
\end{aligned}
\end{equation*}
In view of the irreducibility of $\widetilde{Q}^{k}$ and Yin and Zhang \cite[Lemma A.39]{YinZhang2013},
we obtain $\overline{V}(x,s_{kl})=\overline{V}(x,s_{kr})$ for any $s_{kl},s_{kr}\in \mathcal{M}_{k}$.
\end{proof}
The next theorem shows that the limit function $\overline{V}(x,k)$ is also a viscosity
solution to the limit HJB equation (\ref{limit HJB}). From the uniqueness of viscosity
solution, we have $\overline{V}(x,k)=V^{0}(x,k)$ for $k=1,\ldots,L$. It is mentioned
that the ``strict local maximum (resp., minimum)" will be taken in Definition \ref{VS Definition 1}
to prove the viscosity solution property of $\overline{V}(x,k)$.
\begin{theorem}\label{Convergence theorem}
Let (A1)-(A4) hold. Then, $\overline{V}(x,k)$ is a viscosity solution
to the limit HJB equation (\ref{limit HJB}).
\end{theorem}
\begin{proof}
\emph{Viscosity supersolution property}.
Fix $k=1,\ldots,L$. For any $s_{kl}\in\mathcal{M}_{k}$, let $\overline{V}(x,k)$ be a limit
of $V^{\varepsilon}(x,s_{kl})$ for some subsequence of $\varepsilon$. Take a function
$\varphi(x)\in C^{2}$ such that $\overline{V}(x,k)-\varphi(x)$ has a \emph{strict local minimum}
at $x_{0}$ in a neighborhood $B_{\delta}(x_{0})$. Choose $x^{\varepsilon}_{l}$ such that,
for each $s_{kl}\in\mathcal{M}_{k}$, $V^{\varepsilon}(x,s_{kl})-\varphi(x)$ attains its
\emph{local minimum} at $x^{\varepsilon}_{l}$ in $B_{\delta}(x_{0})$. Then we must have
$x^{\varepsilon}_{l}\rightarrow x_{0}$ as $\varepsilon\rightarrow0$. Moreover, the viscosity
supersolution property of $V^{\varepsilon}$ implies
\begin{equation}\label{Convergence proof 1}
\begin{aligned}
&\sum_{l=1}^{m_{k}}\nu^{k}_{l}\bigg(rV^{\varepsilon}(x^{\varepsilon}_{l},s_{kl})
-b(x^{\varepsilon}_{l},s_{kl})\varphi^{\prime}(x^{\varepsilon}_{l})
+\frac{\theta}{2}\sigma^{2}(x^{\varepsilon}_{l},s_{kl})[\varphi^{\prime}(x^{\varepsilon}_{l})]^{2}
-\frac{1}{2}\sigma^{2}(x^{\varepsilon}_{l},s_{kl})\varphi^{\prime\prime}(x^{\varepsilon}_{l})\\
&-\sum_{j\neq s_{kl}}\lambda_{s_{kl},j}^{\varepsilon}
[V^{\varepsilon}(x^{\varepsilon}_{l},j)-V^{\varepsilon}(x^{\varepsilon}_{l},s_{kl})]
-f(x^{\varepsilon}_{l},s_{kl})\bigg)\geq0,
\end{aligned}
\end{equation}
in which
\begin{equation}\label{Convergence proof 1-1}
\begin{aligned}
&\sum_{l=1}^{m_{k}}\nu^{k}_{l}\sum_{j\neq s_{kl}}\lambda_{s_{kl},j}^{\varepsilon}
[V^{\varepsilon}(x^{\varepsilon}_{l},j)-V^{\varepsilon}(x^{\varepsilon}_{l},s_{kl})]\\
=&\sum_{l=1}^{m_{k}}\nu^{k}_{l}
\bigg(\sum_{r\neq l}\frac{1}{\varepsilon}\widetilde{\lambda}_{s_{kl},s_{kr}}
[V^{\varepsilon}(x^{\varepsilon}_{l},s_{kr})-V^{\varepsilon}(x^{\varepsilon}_{l},s_{kl})]
+\sum_{j\neq s_{kl}}\widehat{\lambda}_{s_{kl},j}
[V^{\varepsilon}(x^{\varepsilon}_{l},j)-V^{\varepsilon}(x^{\varepsilon}_{l},s_{kl})]\bigg).
\end{aligned}
\end{equation}
By noting that
$
V^{\varepsilon}(x^{\varepsilon}_{r},s_{kr})-\varphi(x^{\varepsilon}_{r})
\leq V^{\varepsilon}(x^{\varepsilon}_{l},s_{kr})-\varphi(x^{\varepsilon}_{l}),
$
we have
\begin{equation*}
\begin{aligned}
&\sum_{l=1}^{m_{k}}\nu^{k}_{l}\sum_{r\neq l}\widetilde{\lambda}_{s_{kl},s_{kr}}
[V^{\varepsilon}(x^{\varepsilon}_{l},s_{kr})-V^{\varepsilon}(x^{\varepsilon}_{l},s_{kl})]\\
\geq&\sum_{l=1}^{m_{k}}\nu^{k}_{l}\sum_{r\neq l}\widetilde{\lambda}_{s_{kl},s_{kr}}
[(V^{\varepsilon}(x^{\varepsilon}_{r},s_{kr})-\varphi(x^{\varepsilon}_{r}))
-(V^{\varepsilon}(x^{\varepsilon}_{l},s_{kl})-\varphi(x^{\varepsilon}_{l}))].
\end{aligned}
\end{equation*}
Recall that $\sum_{r=1}^{m_{k}}\widetilde{\lambda}_{s_{kl},s_{kr}}=0$, i.e.,
$\sum_{r\neq l}\widetilde{\lambda}_{s_{kl},s_{kr}}=-\widetilde{\lambda}_{s_{kl},s_{kl}}$,
so the right-hand side of the above inequality is equal to
\begin{equation*}
\begin{aligned}
\sum_{l=1}^{m_{k}}\nu^{k}_{l}\sum_{r=1}^{m_{k}}\widetilde{\lambda}_{s_{kl},s_{kr}}
[V^{\varepsilon}(x^{\varepsilon}_{r},s_{kr})-\varphi(x^{\varepsilon}_{r})]
=\sum_{r=1}^{m_{k}}[V^{\varepsilon}(x^{\varepsilon}_{r},s_{kr})-\varphi(x^{\varepsilon}_{r})]
\sum_{l=1}^{m_{k}}\nu^{k}_{l}\widetilde{\lambda}_{s_{kl},s_{kr}}=0,
\end{aligned}
\end{equation*}
where the last equality is owing to
$\sum_{l=1}^{m_{k}}\nu^{k}_{l}\widetilde{\lambda}_{s_{kl},s_{kr}}=0$
based on the definition of stationary distribution.
Therefore,
\begin{equation}\label{Convergence proof 1-2}
\begin{aligned}
&\sum_{l=1}^{m_{k}}\nu^{k}_{l}
\sum_{r\neq l}\widetilde{\lambda}_{s_{kl},s_{kr}}
[V^{\varepsilon}(x^{\varepsilon}_{l},s_{kr})-V^{\varepsilon}(x^{\varepsilon}_{l},s_{kl})]\geq0.
\end{aligned}
\end{equation}
It follows from (\ref{Convergence proof 1}), (\ref{Convergence proof 1-1}),
and (\ref{Convergence proof 1-2}) that
\begin{equation*}
\begin{aligned}
&\sum_{l=1}^{m_{k}}\nu^{k}_{l}\bigg(r \overline{V}(x_{0},s_{kl})
-b(x_{0},s_{kl})\varphi^{\prime}(x_{0})
+\frac{\theta}{2}\sigma^{2}(x_{0},s_{kl})[\varphi^{\prime}(x_{0})]^{2}
-\frac{1}{2}\sigma^{2}(x_{0},s_{kl})\varphi^{\prime\prime}(x_{0})\\
&-\sum_{j\neq s_{kl}}\widehat{\lambda}_{s_{kl},j}[\overline{V}(x_{0},j)-\overline{V}(x_{0},s_{kl})]
-f(x_{0},s_{kl})\bigg)\\
\geq&\lim_{\varepsilon\rightarrow0}\sum_{l=1}^{m_{k}}\nu^{k}_{l}\bigg(rV^{\varepsilon}(x^{\varepsilon}_{l},s_{kl})
-b(x^{\varepsilon}_{l},s_{kl})\varphi^{\prime}(x^{\varepsilon}_{l})
+\frac{\theta}{2}\sigma^{2}(x^{\varepsilon}_{l},s_{kl})[\varphi^{\prime}(x^{\varepsilon}_{l})]^{2}
-\frac{1}{2}\sigma^{2}(x^{\varepsilon}_{l},s_{kl})\varphi^{\prime\prime}(x^{\varepsilon}_{l})\\
&-\sum_{j\neq s_{kl}}\lambda_{s_{kl},j}^{\varepsilon}
[V^{\varepsilon}(x^{\varepsilon}_{l},j)-V^{\varepsilon}(x^{\varepsilon}_{l},s_{kl})]
-f(x^{\varepsilon}_{l},s_{kl})\bigg)
\geq0.
\end{aligned}
\end{equation*}
By the definition (\ref{limit generator}) of $\overline{Q}$, we have
\begin{equation*}
\begin{aligned}
\sum_{l=1}^{m_{k}}\nu^{k}_{l}\sum_{j\neq s_{kl}}\widehat{\lambda}_{s_{kl},j}
[\overline{V}(x_{0},j)-\overline{V}(x_{0},s_{kl})]
=&\sum_{p\neq k}\overline{\lambda}_{kp}[\overline{V}(x_{0},p)-\overline{V}(x_{0},k)],
\end{aligned}
\end{equation*}
and further, by the definition (\ref{averaged coefficients}) of averaged coefficients, we obtain
\begin{equation*}
\begin{aligned}
&r\overline{V}(x_{0},k)-\overline{b}(x_{0},k)\varphi^{\prime}(x_{0})
+\frac{\theta}{2}\overline{\sigma}^{2}(x_{0},k)[\varphi^{\prime}(x_{0})]^{2}
-\frac{1}{2}\overline{\sigma}^{2}(x_{0},k)\varphi^{\prime\prime}(x_{0})\\
&-\sum_{p\neq k}\overline{\lambda}_{kp}[\overline{V}(x_{0},p)-\overline{V}(x_{0},k)]
-\overline{f}(x_{0},k)\\
=&\sum_{l=1}^{m_{k}}\nu^{k}_{l}\bigg(r \overline{V}(x_{0},s_{kl})
-b(x_{0},s_{kl})\varphi^{\prime}(x_{0})+\frac{\theta}{2}\sigma^{2}(x_{0},s_{kl})[\varphi^{\prime}(x_{0})]^{2}
-\frac{1}{2}\sigma^{2}(x_{0},s_{kl})\varphi^{\prime\prime}(x_{0})\\
&-\sum_{j\neq s_{kl}}\widehat{\lambda}_{s_{kl},j}[\overline{V}(x_{0},j)-\overline{V}(x_{0},s_{kl})]
-f(x_{0},s_{kl})\bigg)
\geq0.
\end{aligned}
\end{equation*}
On the other hand, $\overline{V}(x_{0},k)
=\lim_{\varepsilon\rightarrow 0}V^{\varepsilon}(x_{l}^{\varepsilon},s_{kl})
\geq\lim_{\varepsilon\rightarrow 0}g(x_{l}^{\varepsilon})=g(x_{0})$.
As a consequence, $\overline{V}(x,k)$ is a viscosity supersolution to (\ref{limit HJB}).

\emph{Viscosity subsolution property}.
Let $\varphi(x)\in C^{2}$ be such that $\overline{V}(x,k)-\varphi(x)$ attains its
\emph{strict local maximum} at $x_{0}$ in a neighborhood $B_{\delta}(x_{0})$.
Similarly, choose $x^{\varepsilon}_{l}$ such that, for each $s_{kl}\in\mathcal{M}_{k}$,
$V^{\varepsilon}(x,s_{kl})-\varphi(x)$ attains its \emph{local maximum} at $x^{\varepsilon}_{l}$
in $B_{\delta}(x_{0})$. So we have $x^{\varepsilon}_{l}\rightarrow x_{0}$ as $\varepsilon\rightarrow0$.
We first note that if $\overline{V}(x_{0},k)-g(x_{0})\leq0$, then the viscosity subsolution
property already holds. Otherwise, if $\overline{V}(x_{0},k)-g(x_{0})>0$, then for $\varepsilon$
small enough, we have $V^{\varepsilon}(x_{l}^{\varepsilon},s_{kl})-g(x_{l}^{\varepsilon})>0$.
It follows from the viscosity subsolution property of $V^{\varepsilon}$ that
\begin{equation*}
\begin{aligned}
&rV^{\varepsilon}(x^{\varepsilon}_{l},s_{kl})
-b(x^{\varepsilon}_{l},s_{kl})\varphi^{\prime}(x^{\varepsilon}_{l})
+\frac{\theta}{2}\sigma^{2}(x^{\varepsilon}_{l},s_{kl})[\varphi^{\prime}(x^{\varepsilon}_{l})]^{2}
-\frac{1}{2}\sigma^{2}(x^{\varepsilon}_{l},s_{kl})\varphi^{\prime\prime}(x^{\varepsilon}_{l})\\
&-\sum_{j\neq s_{kl}}\lambda_{s_{kl},j}^{\varepsilon}
[V^{\varepsilon}(x^{\varepsilon}_{l},j)-V^{\varepsilon}(x^{\varepsilon}_{l},s_{kl})]
-f(x^{\varepsilon}_{l},s_{kl})\leq0.
\end{aligned}
\end{equation*}
Then we can repeat the same argument as in the proof of viscosity supersolution property
to show that $\overline{V}(x,k)$ is a viscosity subsolution. Thus, $\overline{V}(x,k)$
is a viscosity solution to (\ref{limit HJB}). From uniqueness of viscosity solution,
we have $\overline{V}(x,k)=V^{0}(x,k)$, $k=1,\ldots,L$.
\end{proof}

\section{Stock selling}\label{Section SS}

In this section, we provide an example of finding the best time to sell a stock. This example will be
numerically solved by computing the optimal selling rule and value function. It turns out that the optimal
selling rule is of \emph{threshold-type}, and the threshold levels depend on the state of the Markov chain.
In addition, we will test the dependence of the solution on the ambiguity factor $\theta$
and the effectiveness of the two-time-scale approximation.

Let the price of the stock be described by a \emph{switching geometric Brownian motion}:
\begin{equation*}
\left\{
\begin{aligned}
dX_{t}=&b(\alpha_{t})X_{t}dt+\sigma(\alpha_{t})X_{t}dB^{\mathbb{P}}_{t},\quad t\geq0,\\
X_{0}=&x\in \mathbb{R},\quad \alpha_{0}=i\in \mathcal{M},
\end{aligned}
\right.
\end{equation*}
where $b(i)$ and $\sigma(i)$, $i\in \mathcal{M}$, are the appreciation rate and volatility rate
of the stock when the market is in state $i$, respectively.

Then, this equation can be rewritten as
\begin{equation*}
\left\{
\begin{aligned}
dX_{t}=&[b(\alpha_{t})X_{t}+\sigma(\alpha_{t})X_{t}q_{t}]dt
+\sigma(\alpha_{t})X_{t}dB^{\mathbb{Q}}_{t},\\
X_{0}=&x\in \mathbb{R},\quad \alpha_{0}=i\in \mathcal{M}.
\end{aligned}
\right.
\end{equation*}
The problem for the investor is to maximize the following reward by selecting
a best time to sell the stock, and the value function is accordingly defined by
\begin{equation*}
\begin{aligned}
V(x,i)=\sup_{\tau\in\mathcal{S}}\inf_{q_{t}\in\mathcal{U}}
\mathbb{E}^{\mathbb{Q}}\bigg[\int_{0}^{\tau}e^{-rt}\frac{q_{t}^{2}}{2\theta}dt
+e^{-r\tau}(X_{\tau}-K)\bigg],
\end{aligned}
\end{equation*}
where $q_{t}^{2}/2\theta$ with $\theta>0$ is the relative entropy to measure the ambiguity
degree (the larger the $\theta$, the bigger the ambiguity degree), and the positive constant
$K$ is the so-called \emph{transaction fee} to sell a stock.

In this example, the HJB equation (\ref{HJB}) reduces to
\begin{equation}\label{numerical HJB}
\begin{aligned}
\min\bigg\{rv(x,i)-b(i)xv^{\prime}(x,i)+\frac{\theta}{2}\sigma^{2}(i)x^{2}[v^{\prime}(x,i)]^{2}
-\frac{1}{2}\sigma^{2}(i)x^{2}v^{\prime\prime}(x,i)-Qv(x,\cdot)(i)&,\\
v(x,i)-(x-K)\bigg\}=0&.
\end{aligned}
\end{equation}
In the following, Subsection \ref{TSC} is concerned with a two-state Markov chain
and Subsection \ref{FSC} focuses on a four-state Markov chain with a two-time-scale structure.

\subsection{Two-state case}\label{TSC}

In this subsection, we consider a two-state Markov chain $\alpha_{t}$ with state space
$\mathcal{M}=\{1,2\}$, whose generator is given by
\begin{equation*}
\begin{aligned}
\left[
  \begin{array}{cc}
    -\lambda_{1} & \lambda_{1} \\
    \lambda_{2} & -\lambda_{2} \\
  \end{array}
\right].
\end{aligned}
\end{equation*}
We take $r=5$, $b(1)=2.125$, $b(2)=0.875$, $\sigma(1)=1$, $\sigma(2)=1$, $K=1$, $\theta=0.01$,
$\lambda_{1}=1$, $\lambda_{2}=1$ as a set of \emph{basic parameters} and numerically compute
the solution to (\ref{numerical HJB}). The step size for $x$ used in the finite difference
scheme is chosen to be 0.01. Under the basic parameters, the optimal selling rule is computed
to be $(x_{1},x_{2})=(1.99,1.53)$ and the value function is plotted in Figure \ref{FigureTwo001}.
The function $x-K$, as the \emph{obstacle part} of (\ref{numerical HJB}), is also plotted as
dotted line; in fact, the threshold levels are exactly the corresponding \emph{intersection points}
of the continuation part and the obstacle part $x-K$.
We see that the price to sell the stock in state 1 is bigger than that in state 2 (i.e., $x_{1}>x_{2}$).
It is reasonable because the appreciation rates $b(1)>b(2)$ and hence the investor
in state 1 would like to sell the stock later than that in state 2 to wait a further going up
of the stock price. Moreover, as expected, $V(x,1)$ is always higher than $V(x,2)$.

In order to demonstrate the influence of ambiguity aversion of the investor, we also compute
the cases when $\theta=0.1$ and $\theta=1$. Together with the basic case when $\theta=0.01$,
the threshold levels $(x_{1},x_{2})$ are listed in Table \ref{Dependence on theta} and
the value functions are plotted in Figures \ref{FigureTwo001}, \ref{FigureTwo01}, \ref{FigureTwo1},
respectively. From Table \ref{Dependence on theta} and Figures \ref{FigureTwo001}, \ref{FigureTwo01},
\ref{FigureTwo1}, it can be found that both $x_{1}$ and $x_{2}$ decrease (i.e., to sell the stock
earlier) and the value function declines (i.e., the reward gets smaller) as $\theta$ increases.
In other words, the investor becomes more \emph{conservative} as the ambiguity degree grows up.

\begin{table}[htbp]
\centering
\caption{Threshold levels}
\label{Dependence on theta}
\renewcommand\arraystretch{1.25}
\begin{tabular}{|p{3cm}<{\centering}|p{3cm}<{\centering}|p{3cm}<{\centering}|p{3cm}<{\centering}|}
  \hline
  $\theta$ & 0.01 & 0.1 & 1  \\
  \hline
  $(x_{1},x_{2})$ & (1.99,1.53) & (1.95,1.52) & (1.69,1.43) \\
  \hline
\end{tabular}
\end{table}

\begin{figure}[htbp]
\centering
\includegraphics[width=5in]{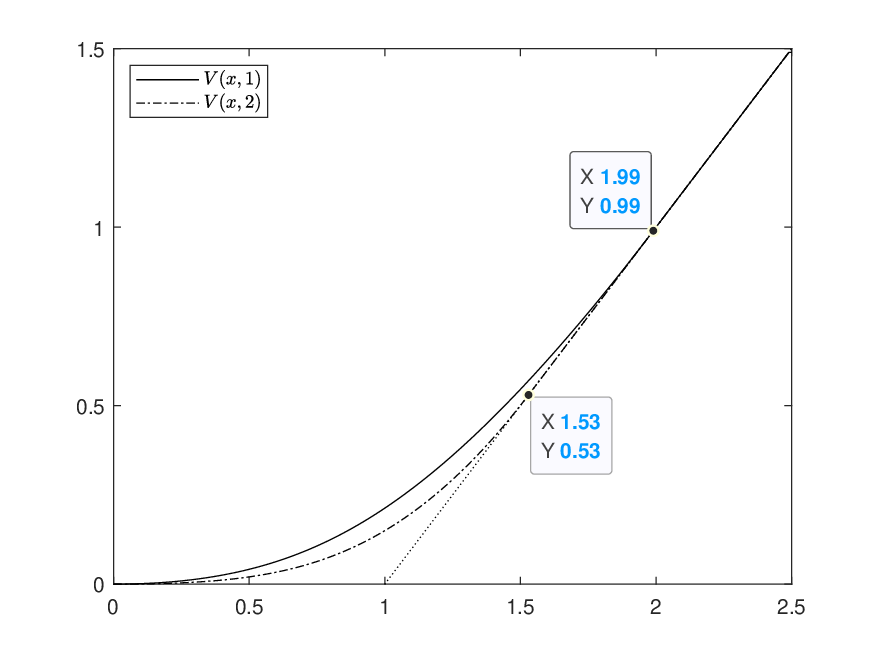}
\caption{Value function when $\theta=0.01$}
\label{FigureTwo001}
\end{figure}

\begin{figure}[htbp]
\centering
\includegraphics[width=5in]{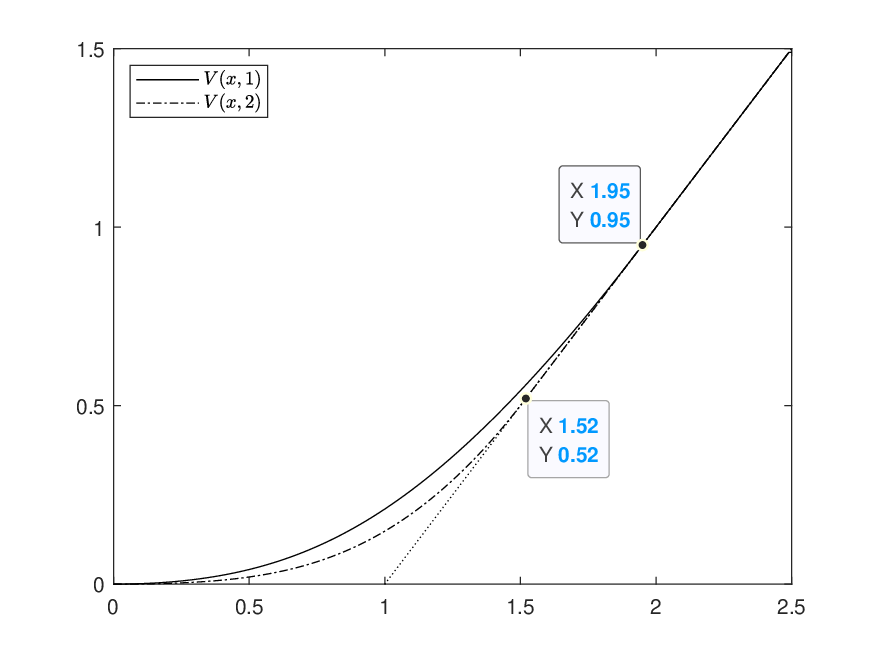}
\caption{Value function when $\theta=0.1$}
\label{FigureTwo01}
\end{figure}

\begin{figure}[htbp]
\centering
\includegraphics[width=5in]{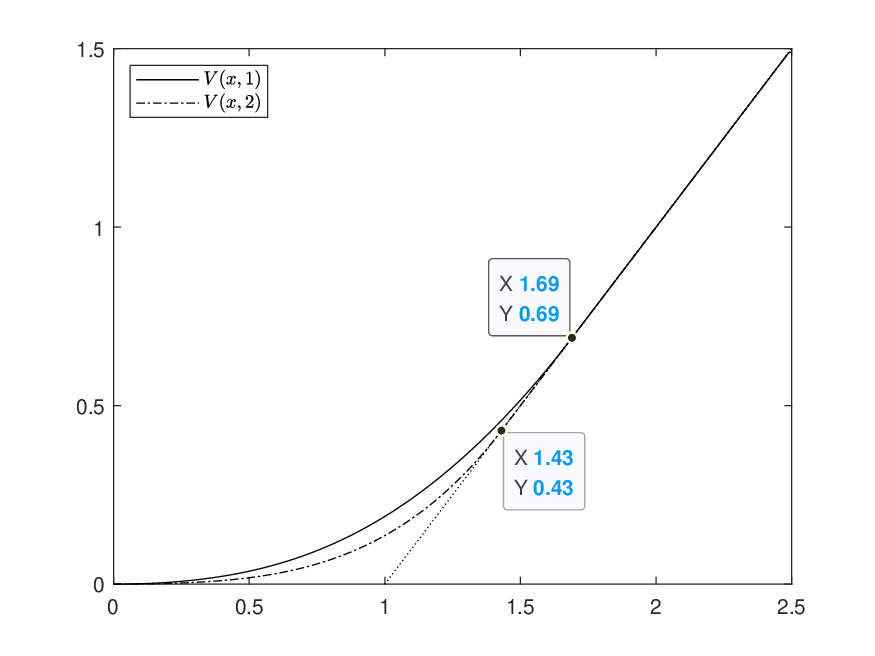}
\caption{Value function when $\theta=1$}
\label{FigureTwo1}
\end{figure}

\subsection{Four-state case}\label{FSC}

In this subsection, the movements of a stock market can be viewed as a composition of
two kinds of \emph{macro trends} that last for a long time and switch to each other
\emph{occasionally} and two kinds of \emph{micro trends} that run for only a short time
and switch to each other \emph{frequently}. To take care of such a situation, one may
consider a four-state Markov chain 
$$
\alpha^{\varepsilon}_{t}=(\alpha_{t}^{1},\alpha_{t}^{2}),
$$
where $\alpha_{t}^{1}\in\{1,2\}$ represents the macro market trends and $\alpha_{t}^{2}\in\{1,2\}$
stands for the micro market trends. We denote the state space of $\alpha^{\varepsilon}_{t}$ as
$$
\{(1,1),(1,2),(2,1),(2,2)\},
$$ 
and assume $\alpha^{\varepsilon}_{t}$ has a \emph{two-time-scale} generator $Q^{\varepsilon}$ given by
\begin{equation*}
\begin{aligned}
\frac{1}{\varepsilon}\left[
                                       \begin{array}{cccc}
                                         -\lambda_{1} & \lambda_{1} & 0 & 0 \\
                                         \lambda_{2} & -\lambda_{2} & 0 & 0 \\
                                         0 & 0 & -\lambda_{1} & \lambda_{1} \\
                                         0 & 0 & \lambda_{2} & -\lambda_{2} \\
                                       \end{array}
                                     \right]
                                     +\left[
                                        \begin{array}{cccc}
                                          -\mu_{1} & 0 & \mu_{1} & 0 \\
                                          0 & -\mu_{1} & 0 & \mu_{1} \\
                                          \mu_{2} & 0 & -\mu_{2} & 0 \\
                                          0 & \mu_{2} & 0 & -\mu_{2} \\
                                        \end{array}
                                      \right].
\end{aligned}
\end{equation*}
For convenience, we also denote the state space of $\alpha^{\varepsilon}_{t}$ as 
$$
\{1,2,3,4\}
$$
with 1=(1,1), 2=(1,2), 3=(2,1), 4=(2,2). We take $b(1)=2.5$, $b(2)=1.75$, $b(3)=1.25$, $b(4)=0.5$,
$\sigma(1)=1$, $\sigma(2)=1$, $\sigma(3)=1$, $\sigma(4)=1$, $\lambda_{1}=2$, $\lambda_{2}=2$,
$\mu_{1}=1$, $\mu_{2}=1$. In addition, $r$, $K$, $\theta(=0.01)$, and the step size for finite
difference scheme are the same with those in Subsection \ref{TSC}. The value functions when
$\varepsilon=1$, $\varepsilon=0.1$, $\varepsilon=0.01$ are plotted in Figures \ref{FigureFour1},
\ref{FigureFour01}, \ref{FigureFour001}, respectively. Note that $V^{\varepsilon}(x,1)$ and
$V^{\varepsilon}(x,2)$ (similarly, $V^{\varepsilon}(x,3)$ and $V^{\varepsilon}(x,4)$) are moving
\emph{closer} as the time-scale parameter $\varepsilon$ approaches 0; they are almost overlapped
in Figure \ref{FigureFour001}.

Finally, we consider the limit problem when $\varepsilon\rightarrow0$. The corresponding
stationary distributions are given by $\nu^{1}=\nu^{2}=(\lambda_{2}/(\lambda_{1}+\lambda_{2}),
\lambda_{1}/(\lambda_{1}+\lambda_{2}))=(1/2,1/2)$, and from (\ref{limit generator}),
the generator for the limit Markov chain $\overline{\alpha}_{t}$ reads
\begin{equation*}
\begin{aligned}
\overline{Q}=\left[
          \begin{array}{cc}
            -\mu_{1} & \mu_{1} \\
            \mu_{2} & -\mu_{2} \\
          \end{array}
        \right]=\left[
                  \begin{array}{cc}
                    -1 & 1 \\
                    1 & -1 \\
                  \end{array}
                \right].
\end{aligned}
\end{equation*}
With the parameters specified above and according to (\ref{averaged coefficients}),
we obtain the averaged coefficients $\overline{b}(1)=2.125$, $\overline{b}(2)=0.875$,
$\overline{\sigma}(1)=1$, $\overline{\sigma}(2)=1$ for the limit problem. Note that
these coefficients are the same with those in the first subsection. So the first
subsection is exactly the limit problem.
The optimal selling rule is given by $(x_{1},x_{2})=(1.99,1.53)$ and the value function
can be found in Figure \ref{FigureTwo001}.

To show the performance of the two-time-scale convergence, we define the \emph{error norms}
$N_{1}^{\varepsilon}$ and $N_{2}^{\varepsilon}$ as
$N_{1}^{\varepsilon}=\sum_{n=1}^{N}(|V^{\varepsilon}_{1}(n)-\overline{V}_{1}(n)|
+|V^{\varepsilon}_{2}(n)-\overline{V}_{1}(n)|)/N$ and
$N_{2}^{\varepsilon}=\sum_{n=1}^{N}(|V^{\varepsilon}_{3}(n)-\overline{V}_{2}(n)|
+|V^{\varepsilon}_{4}(n)-\overline{V}_{2}(n)|)/N$,
where $N$ is the number of meshpoints with respect to $x$, and $V^{\varepsilon}(n)$
(respectively, $\overline{V}(n)$) stands for the value of $V^{\varepsilon}$ (respectively,
$\overline{V}$) at meshpoint $x(n)$. The error norms are calculated and listed in Table \ref{Errors}.
Table \ref{Errors} suggests that $V^{\varepsilon}$ approximates $\overline{V}$ quite well
for small $\varepsilon$. In this sense, the two-time-scale approximation would be helpful
for the \emph{long-term} investors to skip the switching of micro market when it varies too fast,
and only pay attention to the switching of macro market.

\begin{figure}[htbp]
\centering
\includegraphics[width=5in]{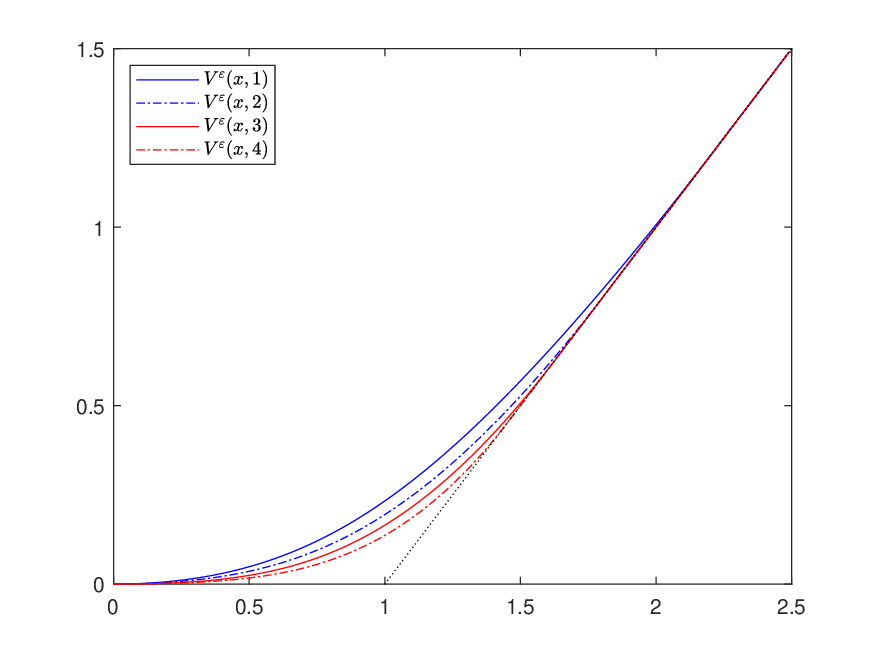}
\caption{Value function when $\varepsilon=1$}
\label{FigureFour1}
\end{figure}

\begin{figure}[htbp]
\centering
\includegraphics[width=5in]{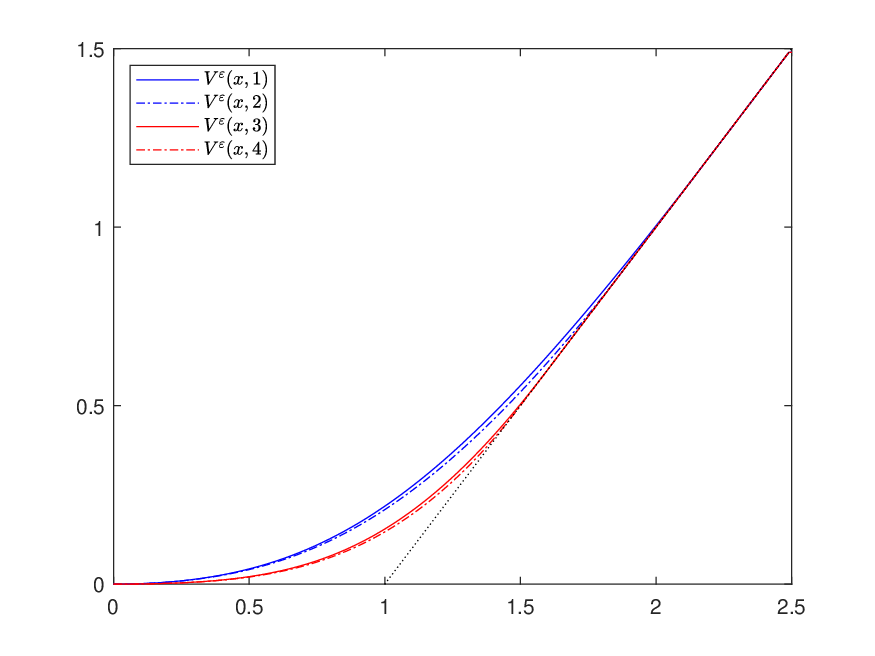}
\caption{Value function when $\varepsilon=0.1$}
\label{FigureFour01}
\end{figure}

\begin{figure}[htbp]
\centering
\includegraphics[width=5in]{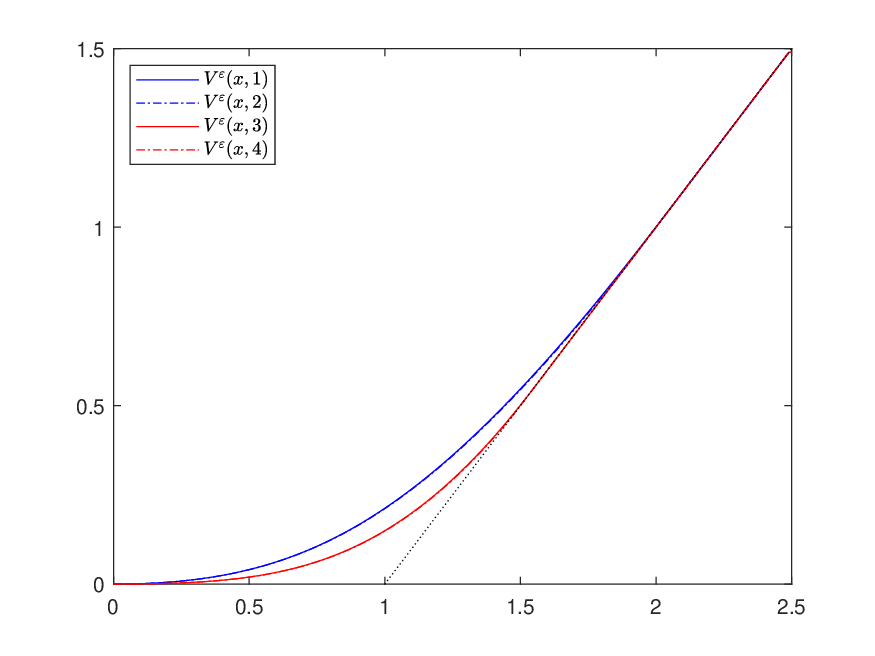}
\caption{Value function when $\varepsilon=0.01$}
\label{FigureFour001}
\end{figure}

\begin{table}[htbp]
\centering
\caption{Error norms}
\label{Errors}
\renewcommand\arraystretch{1.25}
\begin{tabular}{|p{3cm}<{\centering}|p{3cm}<{\centering}|p{3cm}<{\centering}|p{3cm}<{\centering}|}
\hline
$\varepsilon$ & 1 & 0.1 & 0.01 \\
\hline
$N_{1}^{\varepsilon}$ & 0.0200 & 0.0069 & 0.0015 \\
\hline
$N_{2}^{\varepsilon}$ & 0.0091 & 0.0029 & 0.0007 \\
\hline
\end{tabular}
\end{table}

\end{document}